\documentclass[a4paper, 10pt, DIV = 13]{scrartcl}
\usepackage[utf8]{inputenc}
\usepackage[english]{babel}
\usepackage{tikz}
\usetikzlibrary{calc,angles,quotes,arrows.meta}
\usepackage{amsmath,amsfonts,amssymb,amsthm,mathrsfs}
\usepackage{lmodern}
\usepackage{hyperref}
\usepackage{xcolor}
\usepackage{adjustbox}  
\usepackage[normalem]{ulem}
\hypersetup{breaklinks=true}
\usepackage{framed}
\usepackage{mathtools}
\mathtoolsset{showonlyrefs}
\usepackage{longtable}

\usepackage{graphicx}
\usepackage{subcaption}   

\newcommand{\sbullet}{\mathbin{\scalebox{.5}{$\bullet$}}}

\let\origcdot\cdot

\renewcommand*{\cdot}{%
  \mathbin{\vcenter{\hbox{\scalebox{1.3}{$\origcdot$}}}}}

\definecolor{darkred}{rgb}{0.5,0,0}
\definecolor{darkgreen}{rgb}{0,0.5,0}
\definecolor{darkblue}{rgb}{0,0,0.5}
\hypersetup{colorlinks,
linkcolor=darkblue,
filecolor=darkgreen,
urlcolor=darkred,
citecolor=darkblue}

\numberwithin{equation}{section}

\newtheorem{proposition}{Proposition}[section]
\newtheorem{lemma}[proposition]{Lemma}

\newtheorem{claim}[proposition]{Claim}

\theoremstyle{definition}
\newtheorem{theorem}{Theorem}
\newtheorem*{theorem*}{Theorem}
\newtheorem{definition}[proposition]{Definition}
\newtheorem{remark}[proposition]{Remark}

\usepackage{enumitem}
\setlist{nosep}
\setlist{noitemsep}
\setlist{leftmargin=*}

\usepackage{mathtools}

\def\XXint#1#2#3{{\setbox0=\hbox{$#1{#2#3}{\int}$}
     \vcenter{\hbox{$#2#3$}}\kern-.5\wd0}}

\newcommand{\R}{\mathbb{R}}
\newcommand{\Z}{\mathbb{Z}}

\renewcommand{\epsilon}{\varepsilon}
\newcommand{\hal}{\frac{1}{2}}

\newcommand{\dd}{\mathtt{d}}

\newcommand{\CC}{\mathrm{C}}

\newcommand{\E}{\mathbb{E}}

\newcommand{\dist}{\mathrm{dist}}

\newcommand{\La}{\mathsf{L}}

\newcommand{\B}{\mathrm{B}}
\newcommand{\BR}{\B_r}

\renewcommand{\P}{\mathbb{P}}

\renewcommand{\O}{\mathcal{O}}

\newcommand{\Cc}{\mathtt{C}}
\newcommand{\Ccab}{\CC^{ab}}
\newcommand{\Pa}{\Phi_{\alpha}}
\newcommand{\Ea}{\mathrm{E}_{\alpha}}

\newcommand{\AD}{\mathsf{A}_2}
\newcommand{\Ef}{\mathrm{E}_{f}}
\newcommand{\EE}{\mathrm{E}}

\newcommand{\Px}{\mathrm{Q}_{x}}

\newcommand{\hPa}{\widehat{\Ga}}

\newcommand{\X}{\mathsf{X}}
\newcommand{\bX}{\X}

\newcommand{\HH}{\mathcal{H}}

\newcommand{\Ga}{\mathsf{G}_\alpha}

\newcommand{\Psia}{\Psi_\alpha}
\newcommand{\Psiav}{\Psi_{\alpha, v}}
\newcommand{\PsiavU}{\Psi_{\alpha, v_1}}
\newcommand{\PsiavD}{\Psi_{\alpha, v_2}}

\newcommand{\ADD}{\widehat{\mathsf{A}}_2}

\newcommand{\xstar}{r_{\star}}

\newcommand{\PP}{\mathbf{p}}
\newcommand{\pP}{\mathbf{p}}

\newcommand{\per}{\mathrm{per}}
\newcommand{\Psisas}{\Psiav^{\star}}
\newcommand{\Ss}{S^\star}

\newcommand{\bPP}{\PP^{\sbullet}}
\newcommand{\RR}{\mathrm{R}}
\newcommand{\RRab}{\RR^{ab}}
\newcommand{\bPPU}{\PP^{\sbullet (1)}}
\newcommand{\bPPD}{\PP^{\sbullet (2)}}
\newcommand{\Tr}{\mathtt{tr}}
\newcommand{\Cs}{\widehat{\mathsf{R}}}
\newcommand{\Csab}{\widehat{\mathsf{R}}^{ab}}

\newcommand{\alphaL}{\alpha_{\dagger}}
\newcommand{\cc}{\mathtt{c}}
\newcommand{\alphac}{\bar{\alpha}}
\newcommand{\lambdac}{\bar{\lambda}}
\newcommand{\Fm}{\mathrm{F}}
\newcommand{\Gm}{\mathrm{G}}
\newcommand{\Hm}{\mathrm{H}}
\renewcommand{\Im}{\mathrm{I}}

\newcommand{\MM}{\mathrm{M}}
\newcommand{\MMM}{\mathrm{N}}
\newcommand{\TT}{\mathrm{T}}
\newcommand{\alphamax}{\alpha_{m}}
\newcommand{\FF}{\mathcal{F}}
\newcommand{\muf}{\mathrm{W}_f}

\newcommand{\Ws}{\mathrm{W}_s}
\newcommand{\fs}{f_s}

\newcommand{\Low}{\mathrm{Low}}
\newcommand{\MainLow}{\mathrm{MainLow}}
\newcommand{\ErrorLow}{\mathrm{ErrorLow}}

\newcommand{\High}{\mathrm{High}}
\newcommand{\FS}{\mathsf{FS}}
\newcommand{\SM}{\mathsf{SM}}
\newcommand{\LAN}{\Lambda_N}
\newcommand{\PPer}{\PP_{\per}}
\newcommand{\SUM}{\mathsf{Sum}}
 
\newcommand{\FDLT}{\mathsf{F}_{\mathrm{3}}} 

\newcommand{\K}{\mathrm{K}}

\newcommand{\STP}{\mathsf{S}_2^{+}}
\newcommand{\tCs}{\tau \Cs}
\newcommand{\Cst}{\Cs'}
\newcommand{\hPP}{\widehat{\PP}} 
\newcommand{\Mm}{\mathrm{M}}
\newcommand{\mustar}{\mu_{\star}}
\newcommand{\MMMt}{\widetilde{\MMM}}

\newcommand{\trho}{\tilde{\rho}} 
\newcommand{\bkappa}{\bar{\kappa}}
\newcommand{\muh}{\hat{\mu}}
\newcommand{\xss}{r_{\star \star}}

\begin{document}
\title{\LARGE{The hexagonal lattice is universally locally optimal}}
\date{\small{\today}}
\author{\large{Thomas Leblé\thanks{ Université de Paris-Cité, CNRS, MAP5 UMR 8145, F-75006 Paris, France \texttt{thomas.leble@math.cnrs.fr} }}}
\maketitle

\vspace{-0.5cm}
\begin{abstract}
We prove that the hexagonal lattice is a local minimizer, among all point configurations, of the interaction energy per unit volume for pair potentials that are completely monotonic functions of the square distance. This includes Gaussian interactions and power laws.
\end{abstract}
\begin{center}
\textit{Dedicated to Sylvia Serfaty on the occasion of her $50^{th}$ birthday.}
\end{center}

\section{Introduction}
When $f : (0, + \infty) \to \R$ is a function, and $\bX \subset \R^2$ is a locally finite collection of points, define the $f$-energy
\begin{equation}
\label{def:Ef}
\Ef(\bX) := \liminf_{r \to \infty} \frac{1}{|\bX \cap \BR|} \sum_{x, y \in \bX \cap \BR, x \neq y} f(|x-y|),
\end{equation}
where $\BR$ is the disk of center $0$ and radius $r$, while $|\bX \cap \BR|$ is the number of points in $\BR$.

\vspace{0.2cm}

The goal of this paper is to prove that if $f$ is of the form $f : r \mapsto e^{-\pi \alpha r^2}$, or $f : r \mapsto r^{-s}$, then the hexagonal lattice $\AD$ is a \emph{local} minimizer of $\Ef$ among point configurations. Our result extends to all completely monotonic functions of the square distance, see Section \ref{sec:precise_statements_of_our_results}.

An important conjecture states that $\AD$ is in fact the \emph{global} minimizer of those energies at fixed density. However, we are not aware of any local minimality result.

\subsection{The hexagonal lattice}
We denote by $\AD$ the hexagonal lattice $\AD := \sigma \Z  + \tau \Z$, seen as a collection of points in $\R^2$, where:
\begin{equation*}
\sigma := \xstar \times (1,0), \quad \tau := \xstar \times \left(\frac{1}{2}, \frac{\sqrt{3}}{2}\right), \quad \xstar := \sqrt{\frac{2}{\sqrt{3}}}. 
\end{equation*} 
The quantity $\xstar$ is the minimal distance between lattice points. A \emph{fundamental domain} $\HH$ of $\AD$ is given by the Voronoi cell of the origin, i.e. the set of points which are closer to the origin than to any other lattice point, see Figure \ref{figH}. With our choice of $\xstar$, the area of $\HH$ (the “covolume”) is equal to $1$.
\begin{figure}[h!]
\begin{center}
\begin{tikzpicture}[scale=0.6]   
 \pgfmathsetmacro{\xstar}{sqrt(2/sqrt(3))}
  \pgfmathsetmacro{\half}{0.5*\xstar}
  \pgfmathsetmacro{\hhalf}{0.5*\xstar/sqrt(3)}  
  \pgfmathsetmacro{\vfull}{\xstar/sqrt(3)}      
\pgfmathsetmacro{\ytau}{0.5*sqrt(3)*\xstar}

  \draw[fill=gray!25,thick]
    (\half,\hhalf) --
    (0,\vfull) --
    (-\half,\hhalf) --
    (-\half,-\hhalf) --
    (0,-\vfull) --
    (\half,-\hhalf) -- cycle;

  \node at (-0.2,0.18) {$\HH$};

  \draw[-, thick] (0,0) -- (\xstar,0)          node[below] {$\sigma$};
  \draw[-, thick] (0,0) -- (\half,\ytau)       node[right]      {$\tau$};

  \foreach \m in {-2,...,2}{
    \foreach \n in {-2,...,2}{
      \pgfmathsetmacro{\px}{\n*\xstar + \m*0.5*\xstar}
      \pgfmathsetmacro{\py}{\m*0.866025403784*\xstar} 
      \fill (\px,\py) circle(0.04);
    }
  }

  \fill (0,0) circle (0.04);
\end{tikzpicture}
\end{center}
\caption{The hexagonal lattice $\AD$, its basis $(\sigma, \tau)$, and the fundamental domain $\HH$.}
\label{figH}
\end{figure}
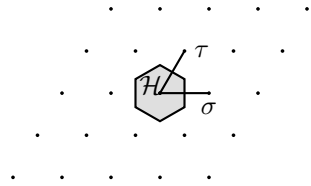

In this paper, “local” is understood with respect to small, bounded perturbations of the lattice: take $\pP : \AD \to \R^2$, and consider the \emph{perturbed lattice} obtained by shifting each lattice point $x$ by $\PP(x)$:
\begin{equation*}
\AD + \pP := \{x + \PP(x), \ x \in \AD\}.
\end{equation*}
We define the size of a perturbation $\PP$ as 
\begin{equation*}
\|\PP\|:= \sup_{x \in \AD} |\PP(x)|.
\end{equation*} 
Whenever $\|\pP\|$ is finite, the resulting point configuration $\AD + \pP$ has density $1$, in the sense that 
\begin{equation}
\label{eq:density}
\lim_{r \to \infty} \frac{1}{|\BR|} \left|(\AD + \pP) \cap \BR\right| = 1.
\end{equation}
In order to ensure that $\AD + \pP$ remains simple, and to shorten some computations, we assume that:
\begin{equation}
\label{eq:pp100}
\|\pP\| \leq \frac{1}{20} \xstar.
\end{equation}
Since we are aiming at \emph{local} results, this is not restrictive. We often write $\PP_x$ instead of $\PP(x)$.

\subsection{Main results}
\label{sec:precise_statements_of_our_results}
For $\alpha > 0$, denote by $\Pa$ the Gaussian interaction potential $\Pa : r \mapsto e^{- \pi \alpha r^2}$. Our first result is:
\begin{theorem}
\label{theo:Local1}
For all $\alpha > 0$, there exists $\epsilon$ (depending on $\alpha$) such that: 
\begin{equation}
\label{eq:Theorem1}
\text{If $\|\pP\| \leq \epsilon$, then $\EE_{\Pa}(\AD + \pP) \geq \EE_{\Pa}(\AD)$.}
\end{equation}
The value of $\epsilon$ can be chosen uniformly for $\alpha$ in compact subsets of $(0, + \infty)$.
\end{theorem}

Following \cite{cohn2007universally}, there is a strong interest in considering more general interaction potentials $f$ that are completely monotonic functions of square distance (we call them here “c.m.s.d”). Recall that a smooth function $g$ is \emph{completely monotonic} when $(-1)^k g^{(k)} \geq 0$ for all $k \geq 0$. A c.m.s.d function is then defined as being of the form $f(r) = g(r^2)$, where $g$ is completely monotonic. It follows from a theorem of Bernstein that every c.m.s.d function $f$ can be written as a mixture of Gaussians, i.e. there exists a certain positive measure $\muf$ on $[0, + \infty)$ such that for all $r > 0$, the following integral is finite and:
\begin{equation}
\label{def:muf}
\int_{0}^{+ \infty} \Pa(r) \dd \muf(\alpha) = f(r).
\end{equation}
\begin{theorem}
\label{theo:Family}
Let $f$ be a c.m.s.d function. There exists $\epsilon > 0$ (depending on $f$) such that: 
\begin{equation}
\label{eq:theofamily}
\text{If $\|\pP\| \leq \epsilon$, then $\Ef(\AD + \pP) \geq \Ef(\AD)$.}
\end{equation}
More generally, let $\FF$ be a family of c.m.s.d. functions, and assume that the following condition is satisfied:
\begin{equation}
\label{condiF}
\lim_{\alpha_0 \to 0, \alpha_1 \to \infty} \frac{ \displaystyle{\int_{\alpha_1}^{+ \infty} e^{- \frac{\pi \alpha}{2} \xstar^2} \dd \muf(\alpha) + \int_{0}^{\alpha_0} \alpha^{-1} \dd \muf(\alpha)}}{\displaystyle{\int_{1}^{+ \infty}  e^{-\pi \alpha \xstar^2} \dd \muf(\alpha) + \int_{0}^{1} e^{-\frac{\pi}{\alpha} \xstar^2} \dd \muf(\alpha)}} = 0 \text{ uniformly for $f \in \FF$.}
\end{equation}
Then $\AD$ is \emph{uniformly locally optimal} for $f \in \FF$, namely there exists $\epsilon > 0$ such that: 
\begin{equation*}
\text{If $\|\pP\| \leq \epsilon$, then $\EE_{f}(\AD + \pP) \geq \EE_{f}(\AD)$ for all $f \in \FF$.}
\end{equation*}
\end{theorem}
\begin{remark} \label{rem:integrability}
When writing \eqref{eq:theofamily}, we implicitly assume $\Ef(\AD) < + \infty$, which, for a c.m.s.d function $f$, is equivalent to $r \mapsto r f(r)$ being integrable at infinity, or to $\alpha \mapsto \alpha^{-1}$ being $\muf$-integrable at $0$, see \eqref{TwoInteg}. Otherwise, it is easy to see that both sides of \eqref{eq:theofamily} are $+ \infty$.
\end{remark}

The power-law (or Riesz) potentials $\fs : r \to r^{-s}$ are c.m.s.d functions whose measure $\Ws$ has density:
\begin{equation}
\label{eq:Ws}
\frac{\dd \Ws}{\dd \alpha} = \frac{\pi^{\frac{s}{2}} \alpha^{\frac{s}{2} - 1}}{\Gamma(\frac{s}{2})} .
\end{equation}
As a direct application of Theorem \ref{theo:Family}, we obtain that $\AD$ is locally optimal for every $\fs$, and that $\epsilon$ can be chosen uniformly for $s$ in compact subsets of $(2, + \infty)$ (when $s \leq 2$, the interaction energies are infinite).

\subsection{Connection with the literature and the “universal optimality” conjecture}
\label{sec:connection_with_the_literature}
The lattice $\AD$ is famously optimal with respect to circle packing. It is moreover known, or expected, to solve many other optimization problems, see \cite{gruber2000many}. This is an instance of “crystallization”, namely the observation that, in various settings, optimal point configurations exhibit a regular structure, see \cite{blanc2015crystallization}. 

In that regard, a major open problem consists in proving that $\AD$ is “universally optimal” for energy minimization, i.e. is a \emph{global} minimizer, among point configurations of density $1$ in the sense of \eqref{eq:density}, of every Gaussian interaction energy, see \cite[Conjecture 9.4]{cohn2007universally} - such interactions are sometimes referred to as the \emph{Gaussian core model} in the physics literature \cite{stillinger1976phase}. By Bernstein's theorem, this would imply \emph{global} minimality for \emph{every} c.m.s.d interaction potential. Partial results are few:
\begin{itemize}
  \item In \cite{montgomery1988minimal}, Montgomery proved that $\AD$ minimizes all those energies \emph{among lattices}. This result has later been generalized in many directions, see e.g. \cite[Theorem 2]{sandier2012ginzburg} for an extension to the Coulomb energy, or \cite{betermin2023maximal} and references therein for variations on the topic of minimization among lattices.
\item Some recent results have gone slightly beyond the lattice case, for instance \cite{hardin2023universally} shows universal optimality of $\AD$ among point configurations with a period of $4$ or $6$ points.
\item  A proof of local optimality among all periodic configurations - with an implicit dependency on the period - had been given in \cite{coulangeon2012energy}, but a flaw was unfortunately later found by the authors \cite{coulangeon2022erratum}. 
\item As put in \cite{faulhuber2024note} \emph{“the status of universal optimality of the hexagonal lattice has been set back to the result of Montgomery from 1988”}.
\end{itemize}
In this paper, we prove \emph{local} minimality of $\AD$ for all c.m.s.d functions, together with a quantitative understanding, in terms of $f$, of:
\begin{enumerate}
  \item The size $\epsilon$ of the neighborhood around $\AD$ on which we guarantee local minimality of $\Ef(\AD)$.
  \item A lower bound, in terms of $\PP$, on $\Ef(\AD + \PP) - \Ef(\AD)$ when $\|\PP\|\leq \epsilon$.
\end{enumerate}
Moreover, we obtain a form of “local universal optimality”, as we show that $\epsilon$ can be chosen uniformly for reasonable families of c.m.s.d functions, see Theorem \ref{theo:Family}. To the best of our knowledge, these are the first results of this kind. However, it is fair to say that:
\begin{itemize}
   \item We remain far from the full conjecture, both qualitatively and quantitatively (our $\epsilon$'s are very small).
   \item The techniques are local by nature and cannot possibly yield a global result.
 \end{itemize} 

\begin{remark}
In dimension $8$ and $24$, both sphere packing and universal optimality were recently solved thanks to major breakthroughs by Viazovska et al. We refer to \cite[Sec.~1]{cohn2022universal} for an introduction to such problems and a survey of what is known, or expected, in other dimensions.
\end{remark}

\subsection{Plan of the paper}
Since c.m.s.d functions are mixtures of the $\Pa$'s, most of the work is devoted to proving, in a quantitative way, that $\AD$ is locally optimal with respect to Gaussian interactions.
\begin{itemize}
  \item In Section \ref{sec:preliminaries}, we introduce some important tools and lemmas. In particular: 
\begin{itemize}
  \item We show that it is enough to consider periodic perturbations (Lemma \ref{lem:Station}).
  \item We present a “periodic” version of the $2$-design property for regular hexagons (Proposition \ref{prop:Random2design}).
  \item We state a key minimality result for a certain lattice sum  (Proposition \ref{prop:minimality}).
\end{itemize}
  \item In Section \ref{sec:proof_of_theorem_ref_theo_local1}, we prove a quantitative version of Theorem \ref{theo:Local1},  then use it to deduce Theorem \ref{theo:Family}.

  \item An appendix is devoted to the proof of auxiliary results. In particular, Appendix \ref{sec:proof_minimality}, which takes up a significant portion of the paper, consists in proving the minimality result of Proposition \ref{prop:minimality}.
\end{itemize}

\noindent \textbf{Acknowledgements.} \textit{We thank Sylvia Serfaty for introducing us to such topics quite some time ago. We thank Martin Huesmann for initial discussions, and Antoine Tilloy for his help with positivity questions.}

\clearpage
\section{Preliminaries}
\label{sec:preliminaries}
We denote the inner product of two vectors $x,y$ of $\R^2$ by $x \cdot y$. 

\subsection{Convention for Fourier analysis on \texorpdfstring{$\R^2$}{the plane}}
\label{sec:conventionFourier}
\paragraph{Fourier transforms.}
We define the Fourier transform $\hat{\varphi}$ of a Schwartz function $\varphi$ as:
\begin{equation*}
\hat{\varphi} : k \mapsto \hat{\varphi}(k) := \int_{\R^2} \varphi(x) e^{-2i\pi x \cdot k} \dd x.
\end{equation*}
With this convention, the Gaussian\footnote{$\Ga$ is a function $\R^2 \to \R$ whereas in Section \ref{sec:precise_statements_of_our_results}, the function $\Pa$ is $[0, + \infty) \to \R$. We have $\Ga(x) = \Pa(|x|)$.} family $\Ga : x \mapsto e^{-\pi \alpha |x|^2}$ on $\R^2$ satisfies 
\begin{equation}
\label{eq:hPa}
\widehat{\Ga} = \alpha^{-1} \mathsf{G}_{\alpha^{-1}} \text{ for $\alpha > 0$.}
\end{equation}

\paragraph{Reciprocal lattice and Pontryagin dual.}
We denote by $\ADD$ the \emph{reciprocal} lattice of $\AD$, defined as:
\begin{equation*}
\ADD := \left\lbrace k \in \R^2, \ k \cdot x \in \Z \text{ for all } x \in \AD \right\rbrace.
\end{equation*}
In practice $\ADD$ corresponds to applying a rotation of angle $\frac{\pi}{2}$ to $\AD$. We denote by $\Omega$ the fundamental domain of $\ADD$ obtained by applying the same rotation to $\HH$. 

\vspace{0.2cm}

The \emph{Pontryagin dual} of $\AD$ is the torus $\R^2 / \ADD$, which can be identified with $\Omega$.

\paragraph{Poisson summation formula.}
If $\varphi : \R^2 \to \R$ is a Schwartz function and $u \in \R^2$, we have:
\begin{equation*}
\sum_{x \in \AD} \varphi(x + u) = \sum_{k \in \ADD} \hat{\varphi}(k) e^{2i\pi k \cdot u}.
\end{equation*}

\subsection{Periodic functions and Fourier transforms.}
We will use sub-lattices of the form $N \AD := N \sigma \Z  + N \tau \Z$ for some integer $N \geq 1$, where $\sigma, \tau$ are the basis vectors of $\AD$. The reciprocal lattice of $N \AD$ is given by 
\begin{equation*}
\widehat{N \AD} = \frac{1}{N} \ADD.
\end{equation*}
We say that a function $g$ defined on $\AD$ is $N\AD$-periodic if 
\begin{equation}
\label{def:periodic}
g(x + y) = g(x) \text{ for all $x \in \AD$ and all $y \in N \AD$}.
\end{equation}
Equivalently\footnote{We will abuse notation and not distinguish between $g$ and the corresponding map on the quotient space.}, $g$ can be seen as function on the discrete torus $\AD / N \AD$. When $g$ is $N\AD$-periodic, we define its Fourier transform as the following function, defined for “frequencies” $k$ in $\frac{1}{N} \ADD / \AD^\star$:
\begin{equation}
\label{Fourier_Lattice}
\hat{g} := k \mapsto \frac{1}{N^2} \sum_{x' \in \AD / N \AD} g(x') e^{-2i\pi k \cdot x'}.
\end{equation}
Then the Plancherel identity holds:
\begin{equation}
\label{eq:Plancherel}
\frac{1}{N^2} \sum_{x' \in \AD / N \AD} |g(x')|^2 =  \sum_{k \in \frac{1}{N} \ADD / \ADD} |\hat{g}(k)|^2.
\end{equation}
We refer e.g. to the lecture notes \cite{Tao} for a pedagogical presentation of this topic.

\subsection{Reduction to periodic perturbations}
\label{sec:reduction_to_stationary_perturbations}
\begin{lemma}
\label{lem:Station}
Let $f$ be a c.m.s.d function. Let $\pP$ be a perturbation such that $\|\PP\| < + \infty$, and let $\delta > 0$. There exists $N \geq 1$ and a perturbation $\PP_{\per}$ which is $N \AD$-periodic in the sense of \eqref{def:periodic}, such that:
\begin{enumerate}
   \item The size of perturbations does not increase: $\|\PP_{\per}\| \leq \|\PP\|$.
   \item The corresponding $f$-energy increases by at most $\delta$: 
   \begin{equation*}
\Ef(\AD + \PP_{\per}) \leq \Ef(\AD + \PP) + \delta.
   \end{equation*}
\end{enumerate}
\end{lemma}
The proof of Lemma \ref{lem:Station} is elementary, we postpone it to Section \ref{sec:proof_periodic}. In order to prove local minimality results, it is thus enough to study periodic perturbations - provided that our analysis \emph{does not depend on the size of the period}. 

\subsection{Properties of periodic perturbations}
\label{sec:properties_of_stationary_perturbations}
Let $\La$ be a sub-lattice of $\AD$, let $\AD / \La$ be the quotient set, which is finite, and let $\PP$ be a $\La$-periodic perturbation.  For $x \in \AD$, we define $\Px$ as the following probability measure on $\R^2$:
\begin{equation}
\label{def:Px}
\Px := \frac{1}{|\AD / \La|} \sum_{x' \in \AD / \La} \delta_{\PP_{x'+x} - \PP_{x'}}.
\end{equation}
This object is crucial for us. It encodes the distribution of “relative displacements” in the direction $x$.
\begin{remark} \label{rem:alwayscentered}
Since $x' \mapsto x' + x \mod \La$ is a permutation of $\AD / \La$, the measure $\Px$ is always centered:
\begin{equation*}
\sum_{x' \in \AD / \La} \PP_{x'+x} - \PP_{x'}  = 0.
\end{equation*}
\end{remark}

\subsubsection*{Energy of $\AD$ plus periodic perturbations.}
For $\alpha > 0$, denote by $\Ea(\bX)$ the $\Pa$-energy of a point configuration $\bX$ in the sense of \eqref{def:Ef}.

\begin{lemma}
If $\PP$ is a periodic perturbation, we have:
\begin{equation}
\label{eq:Ef_Periodic}
\Ea(\AD + \PP) = \sum_{x \in \AD, x \neq 0} \int_{\R^2} \Pa\left(|x + u|\right) \dd \Px(u).
\end{equation}
In particular, 
\begin{equation*}
\Ea(\AD) = \sum_{x \in \AD, x \neq 0} \Pa\left(|x|\right).
\end{equation*}
\end{lemma}
\begin{proof}
This is essentially \cite[Lemma 9.1]{cohn2007universally}, which states\footnote{Their $\Lambda$ is our $\La$, their $N$ is our $|\AD / \La|$, and their points $v$ corresponds to our perturbed lattice points $x + \PP_x$. Their sum runs over $x \in \La$ such that $x + x' - x'' + \PP_{x'} - \PP_{x''} \neq 0$. Since we work under the assumption \eqref{eq:pp100}, and $\PP$ is $\La$-periodic, we have $x + x' - x'' + \PP_{x'} - \PP_{x''} = 0$ iff $x + x' - x'' = 0$}:
\begin{equation*}
\Ea(\AD + \PP) = \frac{1}{|\AD / \La|} \sum_{x', x'' \in \AD / \La} \  \sum_{x \in \La, x + x' - x'' \neq 0} \Pa\left(|x+ x' + \PP_{x'} - x'' - \PP_{x''}|\right).
\end{equation*}
Since $\PP$ is $\La$-periodic we may replace $\PP_{x'}$ by $\PP_{x' + x}$. Moreover, when $x'$ varies in $\AD / \La$ and $x$ varies in $\La$, then $x''' := x+x'$ varies in $\AD$. We thus have:
\begin{equation*}
\Ea(\AD + \PP) = \frac{1}{|\AD / \La|} \sum_{x''' \in \AD} \  \sum_{x'' \in \AD / \La, x''' - x'' \neq 0} \Pa\left(|x''' - x'' + \PP_{x'''} - \PP_{x''}| \right).
\end{equation*}
For $x''$ in $\AD / \La$, when $x'''$ varies in $\AD \setminus \{x''\}$, then $y := x''' - x''$ varies in $\AD \setminus \{0\}$. We may thus write:
\begin{equation*}
\Ea(\AD + \PP) =  \sum_{y \in \AD, y \neq 0} \left(\frac{1}{|\AD / \La|} \sum_{x'' \in \AD / \La} \Pa\left(|y + \PP_{y + x''} - \PP_{x''}| \right)\right), 
\end{equation*}
and by definition of $\Px$, the right-hand side coincides with the right-hand side of \eqref{eq:Ef_Periodic}.
\end{proof}

\subsubsection*{Probabilistic formulation.} 
If $\PP$ is $\La$-periodic for some sub-lattice $\La$, let $\bPP$ be the random family of vectors defined by: 
\begin{equation}
\label{eq:bPP}
\bPP : = (\PP_{x + x^{\sbullet}})_{x \in \AD},
\end{equation} 
where $x^{\sbullet}$ is chosen uniformly at random in the period $\AD / \La$. It is $\AD$-stationary in the sense that:
\begin{equation}
\label{eq:bPP_stat}
\text{For all $z \in \AD$, $\left( \bPP_{x} \right)_{x \in \AD}$ and $\left( \bPP_{x + z} \right)_{x \in \AD}$ have the same finite-dimensional distributions.}
\end{equation}
With this point of view, the measure $\Px$ defined in \eqref{def:Px} represents the law of $\bPP_{x} - \bPP_0$. In particular:
\begin{equation}
\label{eq:QxProbabilistic}
\int_{\R^2} |u|^2 \dd \Px(u) = \E\left[ \left| \bPP_{x} - \bPP_0 \right|^2 \right].
\end{equation}

\subsection{The autocorrelation function}
\label{sec:autocorrelation}
Let $\bPP$ be a random, $\AD$-stationary perturbation in the sense of \eqref{eq:bPP_stat}. For $x \in \AD$, let $\bPPU_x, \bPPD_x$ be the coordinates of $\bPP_x$ in the canonical basis. Assume that $\E \left[ \left|\bPP_0 \right|^2 \right]$ is finite.
\begin{itemize}
\item We define $\CC :  \AD \to \mathcal{M}_2(\R)$ as:
\begin{equation}
\label{eq:defCcx}
\CC : x \mapsto \left( \begin{matrix} \E\left[ \left( \bPPU_x - \bPPU_0 \right)^2 \right] &  \E\left[\left( \bPPU_x - \bPPU_0 \right) \left( \bPPD_x - \bPPD_0 \right) \right] \\ \E\left[\left( \bPPU_x - \bPPU_0 \right) \left( \bPPD_x - \bPPD_0 \right) \right]&  \E\left[ \left( \bPPD_x - \bPPD_0 \right)^2 \right]  \end{matrix}  \right).
\end{equation}
We write $\CC_x$ for the value of the matrix at $x \in \AD$ and denote by $\Ccab_x$ ($a,b \in \{1,2\}$) its coefficients. 
\item We define the \emph{autocorrelation} $\RR  : \AD \to \mathcal{M}_2(\R)$ 
\begin{equation}
\label{eq:defGx}
\RR : x \mapsto \left( \begin{matrix} \E\left[ \bPPU_x  \bPPU_0 \right] &  \E\left[ \hal \left(\bPPU_x  \bPPD_0  + \bPPD_x  \bPPU_0 \right) \right] \\ \E\left[  \hal \left(\bPPU_x  \bPPD_0  + \bPPD_x  \bPPU_0 \right) \right] &  \E\left[ \bPPD_x  \bPPD_0 \right]  \end{matrix}  \right),
\end{equation}
We write similarly $\RR_x$ and $\RRab_x$ for $a,b \in \{1,2\}$.
\end{itemize}
Denote by $^T$ the transpose of a matrix. We have, for all $x \in \AD$:
\begin{equation*}
\CC_x = \E\left[ (\bPP_x - \bPP_0) (\bPP_x - \bPP_0)^{T} \right], \quad \RR_x = \hal \E\left[ \bPP_x (\bPP_0)^{T} + \bPP_0 (\bPP_x)^{T} \right].
\end{equation*}
Using stationarity, a direct computation gives:
\begin{equation}
\label{eq:G0Gx}
\CC_x = 2\left( \RR_0 - \RR_x  \right), \quad \E\left[ \left( \bPP_x - \bPP_0 \right)^2 \right] = \Tr \CC_x = 2 \Tr \left(\RR_0 - \RR_x\right).
\end{equation}
Observe that $\RR$ is a positive semidefinite matrix-valued function in the following sense: for all maps $h : \AD \to \mathbb{C}$ with finite support, the matrix
\begin{equation*}
\sum_{x, y \in \AD} h(x) \overline{h(y)} \RR_{x-y} = \hal \sum_{x, y \in \AD} h(x) \overline{h(y)} \E\left[ \bPP_x (\bPP_y)^{T} + \bPP_y (\bPP_x)^{T} \right] 
\end{equation*}
is positive semi-definite. This has important “spectral” consequences, as we present next.

\subsection{Spectral measure(s)}
\label{sec:Spectral}
Denote by $\STP$ the cone of symmetric positive $2 \times 2$ matrices with real coefficients. A $\STP$-valued measure is a map defined on the Borel $\sigma$-algebra of $\R^2$, with values in $\STP$, satisfying the natural requirements ($\sigma$-additivity etc.) for being a measure, see \cite[Sec. 1.2 \& 1.3]{kimsey2011matrix}. 

\subsubsection*{Existence of the spectral measure.}
The following result generalizes Bochner's theorem on Fourier transforms of positive-definite functions.
\begin{lemma}[\cite{falb1969theorem}]
\label{lem:spectral_measure}
There exists a $\STP$-valued measure $\Cs$, defined on the Pontryagin dual $\Omega$, such that:
\begin{equation}
\label{GasTFrho}
\RR_x = \int_{\Omega} \cos\left(2 \pi x \cdot \omega\right) \dd \Cs(\omega) \text{ for all } x \in \AD.
\end{equation} 
\end{lemma}
The measure $\Cs$ is called the \emph{spectral measure} of the perturbations. For a pedagogical presentation, building on the classical theorem of Bochner, we refer to  \cite[Thm. 2.2.3]{kimsey2011matrix}. We sketch the argument:
\begin{proof}
Both $x \mapsto \RR^{11}_x$ and $x \mapsto \RR^{22}_x$ are positive-definite (real-valued) functions in the usual sense, so by Bochner's theorem their Fourier transform is a finite positive measure on the Pontryagin dual $\Omega$. As for $x \mapsto \RR^{12}_x = \RR^{21}_x$, it is not positive-definite in general, but can easily be written as the difference of two positive-definite functions, thus its Fourier transform is a (signed) measure on $\Omega$, with finite total mass. It remains to check that those four measures put together define a $\STP$-valued measure $\Cs$.
\end{proof}
\begin{remark}
The general result is stated in the Hermitian case, however here $\RR$ is real-symmetric, and $\RR_x = \RR_{-x}$ by stationarity, thus $\RR_x^{ab} = \int_{\Omega} \cos(2 \pi x \cdot \omega) \dd \Csab(\omega)$ and $\RR_x = \int_{\Omega} \cos(2 \pi x \cdot \omega) \dd \Cs(\omega)$.
\end{remark}

\subsubsection*{The trace measure and trace derivative.}
Let $\Csab$ ($a,b \in \{1,2\}$) be the components of $\Cs$, which are (real-valued) finite measures on $\Omega$, corresponding to the Fourier transform of the components of $\RR$. Define the (real-valued) “trace” measure $\tCs$ as:
\begin{equation}
\label{eq:tCs}
\tCs := \Cs^{11} + \Cs^{22},
\end{equation}
which is positive and finite, because so are $\Cs^{11}$ and $\Cs^{22}$. Since $\Cs$ takes values in $\STP$, we know that for all $a, b \in \{1,2\}$, the measure $\Csab$ is absolutely continuous with respect to $\tCs$. Define the “trace derivative” $\Cst$ as the matrix-valued measurable map $\omega \mapsto \Cst(\omega)$, defined on $\Omega$, such that:
\begin{equation*}
\Cst := \frac{\dd \Cs}{\dd \tCs}  \text{ in the sense that } \left(\Cst\right)^{ab} = \frac{\dd \Csab}{\dd \tCs} \text{ for $a,b \in \{1,2\}$.}
\end{equation*}
It follows from \cite{rosenberg1964square} that for $\tCs$-a.e. $\omega$ in $\Omega$, the matrix $\Cst(\omega)$ is in $\STP$ and can be diagonalized\footnote{The choice of $v_1, v_2, \lambda_1, \lambda_2$ can be made in a measurable fashion \cite{wilcox1972measurable}.} in some orthonormal basis $(v_1(\omega), v_2(\omega))$, with eigenvalues $\lambda_1(\omega), \lambda_2(\omega)$ such that 
\begin{equation*}
\lambda_1(\omega) + \lambda_2(\omega) =1.
\end{equation*}
If $f = (f_1, f_2) : \Omega \to \R^2$ is bounded, we can write:
\begin{equation}
\label{quadraticCs}
\sum_{a,b \in \{1,2\}} \int_{\Omega} f_a f_b \dd \Csab = \int_{\Omega} \left(f \cdot \Cst f\right) \dd \tCs = \int_{\Omega} \left(\lambda_1 (f \cdot v_1)^2 + \lambda_2 (f \cdot v_2)^2 \right) \dd \tCs.
\end{equation}
\begin{lemma}
\label{lem:RRORRxIntomegacarre}
We have for all $x \in \AD$:
\begin{equation*}
\int_{\R^2} |u|^2 \dd \Px(u) \leq 8 \pi^2 |x|^2 \Tr\left(\int_{\Omega} |\omega|^2 \dd \Cs(\omega)\right).
\end{equation*}
\end{lemma}
\begin{proof}
Using the definition of $\Cs$ and the elementary inequality $1 - \cos(\theta) \leq \theta^2$, we have:
\begin{equation}
\RR_0 - \RR_x = \int_{\Omega} \left( 1 - \cos(2  \pi x \cdot \omega) \right) \dd \Cs(\omega) \leq 4 \pi^2 \int_{\Omega} |x \cdot \omega|^2 \dd \Cs(\omega) \leq 4 \pi^2 |x|^2 \int_{\Omega} |\omega|^2 \dd \Cs(\omega).
\end{equation}
In particular, since taking $2\Tr(\RR_0 - \RR_x)$ yields the second moment of $\Px$ (see \eqref{eq:G0Gx} and \eqref{eq:QxProbabilistic}), we get: 
\begin{equation}
\label{G0Gxrho}
\int_{\R^2} |u|^2 \dd \Px(u) = 2 \Tr(\RR_0 - \RR_x) \leq 8 \pi^2 |x|^2 \Tr\left(\int_{\Omega} |\omega|^2 \dd \Cs(\omega)\right). 
\end{equation}
\end{proof}

\clearpage

\subsection{A “periodic \texorpdfstring{$2$-design}{2-design}” inequality}
\label{sec:a_random_2_design_property}
\paragraph{Regular hexagons are $2$-designs.}
Each non-zero lattice point is part of a “shell” of six distinct lattice points, obtained by applying a rotation of angle $\frac{k\pi}{3}$ ($k = 0, \dots, 5$), which form the vertices of a regular hexagon.  Regular hexagons enjoy a certain geometric property known as being a \emph{$2$-design}.
\begin{definition}{\cite[Lemma 4.3]{coulangeon2012energy}}
Let $S$ be a finite set of points on the circle of radius $r > 0$, assume that $S$ is symmetric around the origin. It is said to be a \emph{spherical $2$-design} when for all $u \in \R^2$, we have:
\begin{equation}
\label{eq:2design}
\frac{1}{|S|} \sum_{s \in S} (s \cdot u)^2 = \hal r^2 |u|^2.
\end{equation}
\end{definition}
The right-hand side of \eqref{eq:2design} is what one would obtain by averaging $s \mapsto (s \cdot u)^2$ over the circle $r \mathbb{S}^1$ of radius $r$, in other words we have the equality:
\begin{equation}
\label{2designaverage}
\frac{1}{|S|} \sum_{s \in S} (s \cdot u)^2 = \frac{1}{2\pi r} \int_{s \in r \mathbb{S}^1} (s \cdot u)^2 \dd s.
\end{equation}
\begin{remark}[Regular hexagons are $5$-designs]
\label{rem:5design}
For a regular hexagon of sidelength $r$, one can in fact replace the power $2$ in \eqref{2designaverage} by any power up to $5$. Equivalently (see \cite[Definition 4.2 \& Lemma 4.3]{coulangeon2012energy}), for any polynomial $\mathrm{P} : \R^2 \to \R$ of total degree up to $5$, we have:
\begin{equation*}
\frac{1}{|S|} \sum_{s \in S} \mathrm{P}(s) = \frac{1}{2\pi r} \int_{s \in r \mathbb{S}^1} \mathrm{P}(s) \dd s.
\end{equation*}
\end{remark}
\paragraph{A periodic $2$-design inequality.} If $S$ is a “shell” of $\AD$ as described above, and $(\Px)_{x \in \AD}$ are the measures associated to a periodic perturbation as in \eqref{def:Px}, can we ensure that, for some $\cc > 0$:
\begin{equation}
\label{2design_want}
\sum_{s \in S}\int_{\R^2} (s \cdot u)^2 \dd \mathrm{Q}_s(u) \geq \cc \sum_{s \in S} \int_{\R^2}  |u|^2 \dd \mathrm{Q}_s(u) \ ?
\end{equation}
If we had $\mathrm{Q}_s = \mathrm{Q}$ for some $\mathrm{Q}$ independent on $s$, we could exchange the sum and the integral and write:
\begin{equation*}
\sum_{s \in S}  \int_{\R^2} (s \cdot u)^2 \dd \mathrm{Q}_s(u) =  \int_{\R^2}  \left( \sum_{s \in S} (s \cdot u)^2  \right) \dd \mathrm{Q}(u) = \frac{1}{2} r^2 \times |S| \times \int_{\R^2} |u|^2 \dd \mathrm{Q}(u), 
\end{equation*}
using the $2$-design property \eqref{eq:2design}, which would yield \eqref{2design_want}. The problem\footnote{This is, in essence, one of the breaking points in \cite{coulangeon2012energy}, see their erratum \cite{coulangeon2022erratum}.} is that $\mathrm{Q}_s$ \emph{does} depend on $s$. 

The next proposition shows that, by averaging over a period, one recovers a certain inequality reminiscent of the $2$-design property. For a regular hexagon, we sometimes say “radius” instead of sidelength.
\begin{proposition}[Periodic $2$-design inequality]
\label{prop:Random2design}
Let $S$ be a shell of $\AD$ of radius $r > 0$. Let $\PP$ be a perturbation which is $N\AD$-periodic for some $N \geq 1$, in the sense of \eqref{def:periodic}. We have:
\begin{equation}
\label{eq:random2design}
 \sum_{s \in S} \sum_{x' \in \AD / N \AD} \left|s \cdot (\PP_{s+x'} - \PP_{x'})\right|^2 \geq \frac{1}{4} r^2  \sum_{s \in S} \sum_{x' \in \AD / N \AD} \left|\PP_{s+x'} - \PP_{x'}\right|^2.
\end{equation}
\end{proposition}
Using our definition \eqref{def:Px} of the measures $(\Px)_{x \in \AD}$, the inequality \eqref{eq:random2design} (divided by $N^2$) yields \eqref{2design_want} with $\cc = \frac{1}{4} r^2$. The proof of Proposition \ref{prop:Random2design} relies on the following geometric lemma: 
\begin{lemma}
\label{lem:geomwk}
Let $S$ be the vertices of a regular hexagon of radius $r$, let $k$ be in $\R^2$. For $s \in S$, define: 
\begin{equation}
\label{def:wsu}
w_s(k) := 2 \left(1 - \cos\left(2 \pi k \cdot s\right) \right).
\end{equation}
The following inequality holds for all $v$ in $\R^2$:
\begin{equation}
\label{eq:geomwk}
\sum_{s \in S} w_s(k) (s \cdot v)^2 \geq \frac{1}{4} \left(\sum_{s \in S} w_s(k)\right) r^2 |v|^2. 
\end{equation}
\end{lemma}
We postpone the proof of Lemma \ref{lem:geomwk} to Appendix \ref{sec:proof_of_lemma_ref_lem_geomwk}, and we explain how to deduce Proposition \ref{prop:Random2design}.

\begin{proof}[Proof of Proposition \ref{prop:Random2design}] 
Let $\PP$ be a $N\AD$-periodic perturbation and let $\hPP$ be the Fourier transform of $\PP$ as in \eqref{Fourier_Lattice}.  We first fix $s \in S$ and observe that for all frequencies $k \in \frac{1}{N} \ADD / \AD^\star$, we have 
\begin{equation*}
\widehat{\PP_{\cdot+s}}(k) = e^{-2i\pi s \cdot k} \hPP(k).
\end{equation*} 
Dividing both sides of \eqref{eq:random2design} by $N^2$ and applying Plancherel's identity \eqref{eq:Plancherel} yields:
\begin{equation}
\label{Planche1}
\frac{1}{N^2} \sum_{x' \in \AD / N \AD} \left|\PP_{x'+s} - \PP_{x'}\right|^2 = \sum_{k \in \frac{1}{N} \ADD / \AD^\star} \left|e^{-2i\pi s \cdot k} - 1 \right|^2 |\hPP(k)|^2,
\end{equation}
as well as (using the fact that $\widehat{s \cdot \PP}(k) = s \cdot \hPP(k)$ because the vector $s$ is fixed):
\begin{equation}
\label{Planche2}
\frac{1}{N^2} \sum_{x' \in \AD / N \AD} \left|s \cdot (\PP_{x'+s} - \PP_{x'})\right|^2 =  \sum_{k \in \frac{1}{N} \ADD / \AD^\star} \left|e^{-2i\pi s \cdot k} - 1 \right|^2 |s \cdot \hPP(k)|^2.
\end{equation}
Define, as in \eqref{def:wsu}, $w_s(k) := \left|e^{-2i\pi s \cdot k} - 1 \right|^2 = 2 \left(1 - \cos\left(2 \pi k \cdot s\right) \right)$. Inserting \eqref{Planche1} and \eqref{Planche2} into \eqref{eq:random2design}, the desired inequality becomes:
\begin{equation}
\label{eq:random2design_2}
\sum_{k \in \frac{1}{N} \ADD / \AD^\star} \sum_{s \in S} w_s(k) |s \cdot \hPP(k)|^2 \geq \frac{1}{4} r^2 \sum_{k \in \frac{1}{N} \ADD / \AD^\star} \left(\sum_{s \in S} w_s(k)\right) |\hPP(k)|^2.
\end{equation}
Then \eqref{eq:random2design_2}, and thus \eqref{eq:random2design}, follows from applying Lemma \ref{lem:geomwk} for each frequency $k$.
\end{proof}
The constant $\frac{1}{4}$ appears to be asymptotically sharp as $N \to \infty$.
\begin{remark}
\label{rem:squarelattice}
A square is also a spherical $2$-design. However, the square lattice $\Z^2$ - whoses shells are squares - does not satisfy \eqref{2design_want} for any $\cc > 0$.
\end{remark}

\subsection{An auxiliary minimality lemma}
\label{sec:the_minimality_lemma}
Let $v$ be a fixed unit vector in $\R^2$. For all $\alpha > 0$, we consider the function $\Psia : \R^2 \to [0, +\infty)$ defined as:
\begin{equation}
\label{def:Psia}
\Psiav : u \mapsto \sum_{x \in \AD} |(x+u) \cdot v|^2 e^{- \frac{\pi}{\alpha} |x+u|^2}
\end{equation}
This function is clearly $\AD$-periodic, so it is enough to study its restriction to the fundamental hexagon $\HH$. 

\begin{proposition}
\label{prop:minimality}
Let $\alphaL > 0$ be fixed. There exists $\cc > 0$ such that for all $\alpha \in (0, \alphaL)$ and for all $u \in \HH$:
\begin{equation}
\label{eq:PsiaPsia0_statement}
\Psiav(u) - \Psiav(0) \geq \cc |u|^2 e^{- \frac{\pi}{\alpha} \xstar^2}.
\end{equation}
\end{proposition}
The proof of Proposition \ref{prop:minimality} relies on a cumbersome analysis given in Section~\ref{sec:proof_minimality}, including a handful of purely numerical inequalities. We emphasize that $\Psiav$ is \emph{not} convex. The value of $\alphaL$ is irrelevant at this stage and will be chosen later, our argument works for all fixed $\alphaL$ and becomes more difficult for \emph{small} $\alpha$'s.
\begin{remark}
In the sequel, we will encounter the function:
\begin{equation*}
\omega \mapsto \sum_{k \in \ADD} |(k+\omega) \cdot v|^2 e^{- \frac{\pi}{\alpha} |k+\omega|^2}
\end{equation*}
where $\omega$ lives in $\Omega$, the Pontryagin dual of $\AD$. Since the reciprocal lattice $\ADD$ is obtained from $\AD$ by a rotation and $\Omega$ is obtained from $\HH$ by the same rotation, the conclusions of Proposition \ref{prop:minimality} apply readily.
\end{remark}

\clearpage 
\section{Proof of Theorem \ref{theo:Local1}}
\label{sec:proof_of_theorem_ref_theo_local1}
In this section, we denote by $\Ss$ the first shell of $\AD$, of sidelength $\xstar$ (see Figure \ref{figH}), and by $\FS(\PP)$ the “First-Shell” size of a perturbation $\PP$, defined as (with $\Px$ as in \eqref{def:Px}, see also \eqref{eq:QxProbabilistic}):
\begin{equation}
\label{def:FS}
\FS(\PP) := \sum_{x \in \Ss} \int_{\R^2} |u|^2\dd \Px(u) 
\end{equation}
We also define, with $\Cs$ the matrix-valued spectral measure of $\PP$ as introduced in Section \ref{sec:Spectral}:
\begin{equation}
\label{def:SM}
\SM(\PP) := \Tr\left(\int_{\Omega} |w|^2 \dd \Cs(\omega)\right).
\end{equation}
By \eqref{G0Gxrho}, if $\SM(\PP) = 0$ then  $\FS(\PP) = 0$. Moreover if $\FS(\PP) = 0$, then clearly $\PP$ is constant. We prove the following quantitative version of Theorem \ref{theo:Local1}: 
\begin{proposition}
\label{prop:Gaussian_Quantitative}
There exists a threshold $\alphaL > 0$ and a constant $\cc > 0$ such that the following holds.
Let $\PP$ be a periodic perturbation.
\begin{itemize}
  \item For $\alpha \geq \alphaL$, if $\|\PP\| \leq \alpha^{-1}$, then:
  \begin{equation}
\label{QuantitativeAlphaLarge}
\Ea(\AD + \PP) - \Ea(\AD) \geq \cc e^{-\pi \alpha \xstar^2} \times \FS(\PP).
\end{equation}
\item For $0 < \alpha \leq \alphaL$, if $\|\PP\| \leq \cc e^{- \frac{2\pi}{\alpha} \xstar^2}$, then:
  \begin{equation}
\label{QuantitativeAlphaSmall}
\Ea(\AD + \PP) - \Ea(\AD) \geq \cc e^{-\frac{\pi}{\alpha} \xstar^2} \times \SM(\PP).
\end{equation}
\end{itemize}
\end{proposition}
We divide the proof of Proposition \ref{prop:Gaussian_Quantitative} in two parts: “$\alpha$ large” and “$\alpha$ small”. The first case is easier, we simply use the periodic $2$-design inequality of Section \ref{sec:a_random_2_design_property} to get a lower bound on the Hessian of the interaction energy. The second case is more subtle: it requires a distinction between “low” and “high” layers within $\AD$, and the introduction of the spectral measure $\Cs$ after having moved to Fourier side thanks to Poisson's summation formula. The rest of this section is devoted to proving the following two lemmas.
\begin{lemma}[Large $\alpha$]
\label{lem:largealpha}
There exists a threshold $\alphaL \geq 1$, a small constant $\cc$ and a large constant $\Cc$ such that the following holds for all $\alpha \geq \alphaL$ and all periodic perturbations $\PP$:
\begin{itemize}
  \item If $\|\PP\| \leq \frac{1}{\alpha}$, we have:
\begin{equation}
\label{quantitativealphaL1}
\Ea(\AD + \PP) - \Ea(\AD) \geq \cc e^{-\pi \alpha \xstar^2} \times \FS(\PP).
\end{equation}
\item Under our general assumption $\|\PP\| \leq \frac{\xstar}{20}$, see \eqref{eq:pp100}, we have:
\begin{equation}
\label{quantitativealphaL2}
\Ea(\AD + \PP) - \Ea(\AD) \geq - \Cc e^{-\frac{\pi \alpha}{2} \xstar^2} \times  \FS(\PP).
\end{equation}
\end{itemize}
\end{lemma}
\begin{lemma}[Small $\alpha$]
\label{lem:smallalpha}
With $\alphaL$ as above, there exists a small constant $\cc$ and a large constant $\Cc$ such that the following holds for all $0 < \alpha \leq \alphaL$ and all periodic perturbations $\PP$:
\begin{itemize}
  \item If $\|\PP\| \leq \cc e^{- \frac{2\pi}{\alpha} \xstar^2}$, we have:
\begin{equation}
\label{LowerboundAlphaSmall}
\Ea(\AD + \PP) - \Ea(\AD) \geq \cc  e^{-\frac{\pi}{\alpha} \xstar^2} \times \SM(\PP). 
\end{equation}
\item Under our general assumption $\|\PP\| \leq \frac{\xstar}{20}$, see \eqref{eq:pp100}, we have:
\begin{equation}
\label{quantitativealphaSMALL_2}
\Ea(\AD + \PP) - \Ea(\AD) \geq - \Cc \alpha^{-1} \times \SM(\PP).
\end{equation}
\end{itemize}
\end{lemma}
The lower bounds \eqref{quantitativealphaL2} and \eqref{quantitativealphaSMALL_2} are not needed for Proposition \ref{prop:Gaussian_Quantitative} but they follow from the same analysis and will be useful later when considering mixtures of the $\Pa$'s in the proof of Theorem \ref{theo:Family}.

\vspace{0.2cm}

The starting point for both lemmas is \eqref{eq:Ef_Periodic}, namely the fact that, with $\Px$ as in \eqref{def:Px}:
\begin{equation}
\label{EaADPPEaAD}
\Ea(\AD + \PP) - \Ea(\AD) = \sum_{x \in \AD \setminus \{0\}} \int_{\R^2} \left(\Ga(x +u) - \Ga(x) \right) \dd \Px(u), 
\end{equation}
where $\Ga : \R^2 \to \R$ is the Gaussian interaction potential $\Ga(x) = \Pa(|x|) = e^{-\pi \alpha |x|^2}$.

\clearpage

\subsection{First case: \texorpdfstring{$\alpha$}{alpha} large}
Without loss of generality, we can take here $\alpha \geq 1$. We start with the proof of \eqref{quantitativealphaL1} and assume that $\PP$ is bounded by $\frac{1}{\alpha}$.

\paragraph{Step 1. Taylor expansion.}
In view of \eqref{EaADPPEaAD}, for fixed $x \in \AD$ and $u$ in $\R^2$, we use a second-order Taylor expansion:
\begin{equation*}
\Ga(x +u) - \Ga(x) = \nabla \Ga(x) \cdot u + \int_{0}^{1} (1-t) \nabla^2 \Ga(x + t u)(u, u) \dd t, 
\end{equation*}
and we integrate this identity against $\Px$. The first-order term vanishes thanks to Remark \ref{rem:alwayscentered}, and thus:
\begin{equation}
\label{TaylorIntegral}
\int_{\R^2} \left(\Ga(x +u) -  \Ga(x)\right) \dd \Px(u) = \int_{0}^{1} (1-t) \left(\int_{\R^2}  \nabla^2 \Ga(x + tu)(u, u)  \dd \Px(u) \right) \dd t.
\end{equation}
For any $z, u$ in $\R^2$, a direct computation gives:
\begin{equation}
\label{nab2pa}
\nabla^2 \Ga(z)(u,u) = 2 \pi \alpha \left( 2\pi \alpha \left| z \cdot u \right|^2 - |u|^2 \right) e^{-\pi \alpha |z|^2}.
\end{equation}
and then for $t \in [0,1]$, applying \eqref{nab2pa} to $z = x + tu$ yields:
\begin{equation}
\label{nab2papplied}
\nabla^2 \Ga(x + tu) \left(u,u\right) = 2 \pi \alpha \left( 2\pi \alpha \left| (x + tu) \cdot u \right|^2 - |u|^2 \right) e^{-\pi \alpha |x + tu|^2},
\end{equation}
Expanding the squares, using $t x \cdot u \geq - |x \cdot u|$, and discarding the non-negative $2 \pi \alpha t^2 |u|^4$ term, we get:
\begin{multline}
\label{expandingsquaresTaylor}
\left( 2\pi \alpha \left| (x + tu) \cdot u \right|^2 - |u|^2 \right) e^{-\pi \alpha |x + tu|^2} \\ \geq \left( 2\pi \alpha |x \cdot u|^2 - 4 \pi \alpha |x \cdot u| |u|^2 - |u|^2 \right) e^{-\pi \alpha |x|^2 - \pi \alpha t^2 |u|^2 - 2 \pi t \alpha x \cdot u}.
\end{multline}
Since we assume that $\|\PP\| \leq \frac{1}{\alpha}$, the measure $\Px$ is supported on the disk $\{u, |u| \leq \frac{2}{\alpha}\}$. For $|u| \leq \frac{2}{\alpha}$, we have, for some $\Cc$ independent of $\alpha \geq 1$ and of $x, t, u$:
\begin{equation}
\label{eq:factorcontrol}
 4 \pi \alpha |x \cdot u| |u|^2 \leq \pi \alpha |x \cdot u|^2  + \Cc |u|^2
\end{equation}
using Young's inequality first, then our assumption on $|u|$. For the same reasons:
\begin{equation}
\label{eq:exponentcontrol}
0 \leq \pi \alpha t^2 |u|^2 \leq \Cc, \quad - \Cc |x| \leq 2 \pi t \alpha x \cdot u \leq \Cc |x|.
\end{equation}
Using \eqref{eq:factorcontrol}, we get:
\begin{equation*}
2\pi \alpha |x \cdot u|^2 - 4 \pi \alpha |x \cdot u| |u|^2 - |u|^2 \geq \pi \alpha |x \cdot u|^2  - (\Cc+1) |u|^2,
\end{equation*}
and thus, returning to \eqref{expandingsquaresTaylor} and controlling the exponent by its worst/best value using \eqref{eq:exponentcontrol}:
\begin{equation}
\label{BestWorst}
\left( 2\pi \alpha \left| (x + tu) \cdot u \right|^2 - |u|^2 \right) e^{-\pi \alpha |x + tu|^2} \geq \pi \alpha |x \cdot u|^2 e^{-\pi \alpha |x|^2 - \Cc |x| - \Cc} - (\Cc+1)|u|^2  e^{-\pi \alpha |x|^2 + \Cc |x|}.
\end{equation}

\paragraph{Step 2. A lower bound for the first shell.}
We focus on the first shell of the lattice $\Ss$, of radius $\xstar$.
\begin{claim}
\label{claim:AlphaLargeFirstShell}
There exists $\cc > 0$ such that, if $\alpha$ is large enough:
\begin{equation}
\label{AlphaLargeFirstShell}
\sum_{x \in \Ss} \int_{\R^2} \left( \pi \alpha |x \cdot u|^2 e^{-\pi \alpha |x|^2 - \Cc |x| - \Cc} - (\Cc+1)|u|^2  e^{-\pi \alpha |x|^2 + \Cc |x|}  \right) \dd \Px(u) 
\geq \cc \alpha e^{-\pi \alpha \xstar^2} \FS(\PP).
\end{equation}
\end{claim}
\begin{proof}
Using first Proposition \ref{prop:Random2design}, then the definition \eqref{def:FS} of $\FS$, we have:
\begin{multline*}
\sum_{x \in \Ss} \int_{\R^2} \pi \alpha  |x \cdot u|^2 e^{-\pi \alpha |x|^2 - \Cc |x| - \Cc} \dd \Px(u) = \pi \alpha e^{-\pi \alpha \xstar^2 - \Cc \xstar - \Cc} \times \left(\sum_{x \in \Ss} \int_{\R^2} |x \cdot u|^2\dd \Px(u)\right) \\  \geq  \pi \alpha e^{-\pi \alpha \xstar^2 - \Cc \xstar - \Cc} \times \frac{1}{4} \left(\sum_{x \in \Ss}  \int_{\R^2} |u|^2\dd \Px(u)\right) = \left(\frac{\pi}{4} e^{- \Cc \xstar - \Cc}\right) \alpha e^{-\pi \alpha \xstar^2} \FS(\PP).
\end{multline*}
Write $\left(\frac{\pi}{4} e^{- \Cc\xstar - \Cc}\right)$ as some positive constant $\cc$. The left-hand side of \eqref{AlphaLargeFirstShell} is thus bounded below by:
\begin{equation*}
\left(\cc \alpha  - (\Cc+1)e^{\Cc \xstar} \right) e^{-\pi \alpha \xstar^2} \FS(\PP).
\end{equation*}
For $\alpha$ large enough, we have $\cc \alpha  - (\Cc+1)e^{\Cc \xstar} \geq \hal \cc \alpha$, which concludes the proof of Claim \ref{claim:AlphaLargeFirstShell}.
 \end{proof}
This last step only works for a fixed radius of the shell. For “higher” shells, we use a rough lower bound.

\paragraph{Step 3. A lower bound for higher shells.}
\begin{claim}
If $x \in \AD \setminus \{0\}$, $t \in [0,1]$, and $\Px$ is supported on $\{u, |u| \leq \epsilon \}$ for some $\epsilon > 0$, we have:
\begin{equation}
\label{RoughLBx}
\int_{\R^2} \left( 2\pi \alpha \left| (x + tu) \cdot u \right|^2 - |u|^2 \right) e^{-\pi \alpha |x + tu|^2} \dd \Px(u) \geq -   e^{-\pi \alpha |x|^2 + 2 \pi \alpha |x| \epsilon} \left( \int_{\R^2} |u|^2 \dd \Px(u)  \right).
\end{equation}
\end{claim}
\begin{proof}
We simply discard the positive contribution, and write: $e^{-\pi \alpha |x + tu|^2} \leq e^{-\pi \alpha |x|^2 + 2 \pi \alpha |x| \epsilon}$.
\end{proof}
Since we assume here that $\|\PP\|$ is smaller than $\frac{1}{\alpha}$, we can take $\epsilon = \frac{2}{\alpha}$ in \eqref{RoughLBx} and get:
\begin{equation}
\label{RoughLBx2}
\int_{\R^2} \left( 2\pi \alpha \left| (x + tu) \cdot u \right|^2 - |u|^2 \right) e^{-\pi \alpha |x + tu|^2} \dd \Px(u) \geq -   e^{-\pi \alpha |x|^2 + 4 \pi |x|} \left( \int_{\R^2} |u|^2 \dd \Px(u)  \right).
\end{equation}
Combining \eqref{EaADPPEaAD}, \eqref{TaylorIntegral}, \eqref{nab2papplied}, \eqref{BestWorst}, \eqref{AlphaLargeFirstShell} and \eqref{RoughLBx2}, and integrating over $t$, we find that, for some small constant $\cc > 0$, provided $\alpha$ is large enough, the quantity $\Ea(\AD + \PP) - \Ea(\AD)$ is bounded below by:
\begin{equation}
\label{TWoSumsAlpha}
2\pi \alpha \left(\cc \alpha  e^{-\pi \alpha \xstar^2} \FS(\PP) - \sum_{x \in \AD, |x| > \xstar}  e^{-\pi \alpha |x|^2 + 4 \pi |x|}  \left( \int_{\R^2} |u|^2 \dd \Px(u)  \right) \right).
\end{equation}
We want to argue that, for $\alpha$ large enough, the second sum in \eqref{TWoSumsAlpha} is dominated by the first one because of the quickly decreasing Gaussian weight. However, we also need to compare the values of $ \int_{\R^2} |u|^2 \dd \Px(u)$ for $x$ in the first shell, which contribute to $\FS(\PP)$ (see \eqref{def:FS}), and its values in “higher” shells.

\paragraph{Step 4. Comparing relative displacements far and close to the origin.}
\begin{claim}
\label{claim:HigherComparedToSS}
For $x \in \AD$, with $|x| > \xstar$, we have:
\begin{equation}
\label{eq:HigherComparedToSS}
\int_{\R^2} |u|^2 \dd \Px(u) \leq \left(\frac{|x|}{\xstar}\right)^2 \FS(\PP). 
\end{equation}
\end{claim}
\begin{proof}
This is easy to see with the probabilistic formulation, for which (see \eqref{eq:QxProbabilistic}):
\begin{equation*}
\int_{\R^2} |u|^2 \dd \Px(u) = \E\left[ |\bPP_{x} - \bPP_{0}|^2 \right],
\end{equation*} 
where $\bPP$ is the $\AD$-stationary random perturbation defined in \eqref{eq:bPP}. Pick a lattice point $x \in \AD$, there exists $n \leq \frac{|x|}{\xstar}$ and a path $0 = x_1, \dots, x_n = x$ within $\AD$ such that $x_{i+1} - x_i$ belongs to the first shell $\Ss$ for all $i = 0, \dots, n-1$. We have, using Cauchy-Schwarz, the stationarity of $\bPP$, and a rough upper bound:
\begin{multline*}
\E\left[ |\bPP_{x} - \bPP_{0}|^2 \right]= \E\left[ \left|\sum_{i=0}^{n-1} \bPP_{x_{i+1}} - \bPP_{x_{i}}\right|^2 \right] \leq n \sum_{i=0}^{n-1} \E\left[ \left|\bPP_{x_{i+1}} - \bPP_{x_{i}}\right|^2 \right] \\ = n \sum_{i=0}^{n-1} \E\left[ \left|\bPP_{x_{i+1} - x_i} - \bPP_{0}\right|^2 \right] \leq n^2 \sum_{x \in \Ss}\E\left[ \left|\bPP_{x} - \bPP_{0}\right|^2 \right].
\end{multline*}
 \end{proof}

\paragraph{Step 5. Proof of \eqref{quantitativealphaL1}.}
Inserting \eqref{eq:HigherComparedToSS} into \eqref{TWoSumsAlpha}, we find:
 \begin{equation*}
\Ea(\AD + \PP) - \Ea(\AD) \geq 2\pi \alpha \left(\cc \alpha e^{-\pi \alpha \xstar^2 } - \frac{1}{\xstar^2} \sum_{x \in \AD, |x| > \xstar} |x|^2 e^{-\pi \alpha |x|^2 + 4 \pi |x|} \right) \FS(\PP). 
 \end{equation*}
As $\alpha \to \infty$, we have $\sum_{x \in \AD, |x| > \xstar} |x|^2 e^{-\pi \alpha |x|^2 + 4 \pi |x|} = o(e^{-\pi \alpha \xstar^2 })$ which concludes the proof of \eqref{quantitativealphaL1} with an extra factor $\alpha^2$, which we discard for clarity (here $\alpha$ is bounded below by $1$). 

This yields local optimality of $\AD$ with respect to Gaussian potentials that are “very peaked”, and are thus essentially nearest-neighbor interactions. The main ingredient is our “periodic $2$-design inequality”, which basically leverages the lattice symmetry to ensure that some Hessian is positive. 

\vspace{0.2cm}

We conclude the proof of Lemma \ref{lem:largealpha} by proving \eqref{quantitativealphaL2}. We now only assume that $\|\PP\|$ satisfies \eqref{eq:pp100}.

\paragraph{Step 6. A general lower bound.}
We proceed as above: for each $x \in \AD \setminus\{0\}$, use the Taylor expansion \eqref{TaylorIntegral} and insert the expression \eqref{nab2papplied} for the second derivative. Then, discard the positive contribution and simply use the rough lower bound \eqref{RoughLBx} with $\epsilon = 2 \|\PP\| \leq \frac{\xstar}{10}$. We get:
\begin{equation}
\label{prequantalphaL2}
\Ea(\AD + \PP) - \Ea(\AD) \geq - \Cc \alpha \sum_{x \in \AD, x \neq 0} e^{-\pi \alpha |x|^2 + \frac{1}{5} \pi \alpha |x| \xstar } \left( \int_{\R^2} |u|^2 \dd \Px(u)  \right).
\end{equation}
Inserting \eqref{eq:HigherComparedToSS}, we can actually write this lower bound in terms of $\FS(\PP)$:
\begin{equation*}
\Ea(\AD + \PP) - \Ea(\AD) \geq - \Cc \alpha \left(\sum_{x \in \AD, x \neq 0} |x|^2 e^{-\pi \alpha |x|^2 + \frac{1}{5}  \pi \alpha |x| \xstar}\right) \FS(\PP).
\end{equation*}
For $\alpha$ large enough, the sum is dominated by the contribution of the first shell, which yields
\begin{equation*}
\Ea(\AD + \PP) - \Ea(\AD) \geq - \Cc \alpha e^{-\pi \alpha \xstar^2 + \frac{1}{5}  \pi \alpha \xstar^2} \times \FS(\PP),
\end{equation*}
and since $\alpha e^{-\pi \alpha \xstar^2 + \frac{1}{5}  \pi \alpha \xstar^2} = o(e^{- \frac{\pi \alpha}{2} \xstar^2})$ as $\alpha \to \infty$, we get \eqref{quantitativealphaL2}.

\vspace{0.2cm}

In conclusion, we can choose $\alphaL \geq 1$ large enough such that \eqref{quantitativealphaL1} and \eqref{quantitativealphaL2} hold for $\alpha \geq \alphaL$.

\subsection{Second case: \texorpdfstring{$\alpha$}{alpha} small.}
We now assume that $\alpha \in (0, \alphaL)$, where $\alphaL$ is the threshold fixed above. We start with the proof of \eqref{LowerboundAlphaSmall}.

\paragraph{Step 0. Decomposition into low and high layers.}
Let $R  \geq 0$ to be chosen later. We return to \eqref{EaADPPEaAD} and split the lattice sum into “low” and “high” layers, namely we write $
\Ea(\AD + \PP) - \Ea(\AD) = \Low + \High$:
\begin{equation*}
\Low :=  \sum_{x \in \AD, |x| \leq R} \int_{\R^2} \left(\Ga(x+u)  - \Ga(x)\right) \dd \Px(u), \quad \High :=  \sum_{x \in \AD, |x| > R} \int_{\R^2} \left(\Ga(x+u)  - \Ga(x)\right) \dd \Px(u).
\end{equation*}
We start with the $\Low$ term (we can include $x = 0$ in the sum, as its contribution is zero since $\mathrm{Q}_0 = \delta_0$).

\paragraph{Step 1: Taylor expansion to third order.} For all $x, u$ we can write:
\begin{equation*}
\Ga(x+u) = \Ga(x) + \nabla \Ga(x) \cdot u + \hal \nabla^2 \Ga(x)(u,u) + \O(|u|^3) \|\nabla^3 \Ga\|_{\infty}. 
\end{equation*}
By scaling, we have $\|\nabla^3 \Ga\|_{\infty} = \Cc \alpha^{\frac{3}{2}}$. Integrating this expansion against $\Px$ yields, for some $\Cc >0$:
\begin{equation*}
\int_{\R^2} \Ga(x+u) \dd \Px(u) \geq \Ga(x) + \hal \int_{\R^2} \nabla^2 \Ga(x)(u,u) \dd \Px(u) - \Cc \|\PP\| \alpha^{\frac{3}{2}} \int_{\R^2} |u|^2 \dd \Px(u)
\end{equation*}
where we used the fact that $\Px$ is centered, and supported on a disk of radius $2 \|\PP\|$. We may thus write:
\begin{equation}
\label{MainLow1}
\Low
= \hal \sum_{x \in \AD, |x| \leq R} \int_{\R^2} \nabla^2 \Ga(x)(u,u) \dd \Px(u) + \ErrorLow_1,
\end{equation}
with an error term satisfying:
\begin{equation}
\label{ErrorLow1}
 \ErrorLow_1 \geq - \Cc \|\PP\| \alpha^{\frac{3}{2}} \sum_{x \in \AD, |x| \leq R} \int_{\R^2} |u|^2 \dd \Px(u).
\end{equation}

\paragraph{Step 2: Re-introducing the missing layers.} Using \eqref{nab2pa}, we know that for all $x,u$ we have:
\begin{equation}
\nabla^2 \Ga(x)(u,u) = 2 \pi \alpha \left( 2\pi \alpha \left| x \cdot u \right|^2 - |u|^2 \right) e^{-\pi \alpha |x|^2} \leq 4 \pi^2 \alpha^2 |x|^2 |u|^2 e^{-\pi \alpha |x|^2}.
\end{equation}
We now simply re-introduce the missing layers in the sum over $x$ in \eqref{MainLow1} and write:
\begin{multline}
\sum_{x \in \AD, |x| \leq R} \int_{\R^2} \nabla^2 \Ga(x)(u,u) \dd \Px(u) = \sum_{x \in \AD} \int_{\R^2} \nabla^2 \Ga(x)(u,u) \dd \Px(u) - \sum_{x \in \AD, |x| > R} \int_{\R^2} \nabla^2 \Ga(x)(u,u) \dd \Px(u) \\ \geq \sum_{x \in \AD} \int_{\R^2} \nabla^2 \Ga(x)(u,u) \dd \Px(u) - 4 \pi^2 \alpha^2 \sum_{x \in \AD, |x| > R} |x|^2 e^{-\pi \alpha |x|^2} \left(\int_{\R^2} |u|^2 \dd \Px(u)\right).
\end{multline}
Combining this with \eqref{MainLow1}, \eqref{ErrorLow1}, we write:
\begin{equation}
\Low = \MainLow + \ErrorLow_2,
\end{equation}
where $\MainLow$ is the full lattice sum
\begin{equation}
\label{def:MainLow}
\MainLow := \hal \sum_{x \in \AD} \int_{\R^2} \nabla^2 \Ga(x)(u,u) \dd \Px(u), 
\end{equation}
and the error term is now such that:
\begin{equation}
\label{ErrorLow2}
 \ErrorLow_2 \geq - \Cc \|\PP\| \sum_{x \in \AD, |x| \leq R} \int_{\R^2} |u|^2 \dd \Px(u) - \Cc  \sum_{x \in \AD, |x| > R} |x|^2 e^{-\pi \alpha |x|^2} \left(\int_{\R^2} |u|^2 \dd \Px(u)\right).
\end{equation}
\begin{remark}
In \eqref{ErrorLow2} we discarded some powers of $\alpha$ ($\alpha^{\frac{3}{2}}$ in front of the first sum and $\alpha^2$ in front of the second one), and blended them into the multiplicative constant, which is valid because $\alpha \leq \alphaL$. Those prefactors would only have a minor effect, as ultimately the main barrier to overcome is exponential in $\frac{1}{\alpha}$, so we prefer to not keep track of them.
\end{remark}

In the next step, we focus on $\MainLow$ and move to Fourier side using Poisson summation's formula.

\paragraph{Step 3: Introducing the autocorrelation and the spectral measure(s).}
For all $x,u$, write $u = (u_1, u_2)$ and $\nabla^2 \Ga(x)(u,u)$ as:
\begin{equation}
\label{nabla2pab}
  \nabla^2 \Ga(x)(u,u) = \sum_{a,b \in \{1,2\}} \partial^2_{ab} \Ga(x) u_a u_b,  
\end{equation} 
and observe that, with the notation introduced in \eqref{eq:defCcx}, \eqref{eq:G0Gx}, we have:
\begin{equation}
\label{uaubPxCcab}
  \int_{\R^2} u_a u_b \dd \Px(u) = \Ccab = 2 \left[\RRab_0 - \RRab_x \right],
\end{equation} 
where $\RR$ is the “autocorrelation” of $\PP$. Combining \eqref{nabla2pab}, \eqref{uaubPxCcab} and \eqref{def:MainLow}, we may thus write:
\begin{equation}
  \MainLow = \sum_{a,b \in \{1,2\}} \sum_{x \in \AD}  \partial^2_{ab} \Ga(x) \left[\RRab_0 - \RRab_x \right].
\end{equation}
Since $\Ga$ is a Schwartz function, so are its derivatives. Moreover the autocorrelation is bounded, so we may apply Poisson's summation formula, which gives, for each $a,b \in \{1,2\}$:
\begin{equation}
  \sum_{x \in \AD}  \partial^2_{ab} \Ga(x) \left[\RRab_0 - \RRab_x \right] = \sum_{k \in \ADD} \widehat{\partial^2_{ab} \Ga}(k) \times \RRab_0 -  \widehat{\partial^2_{ab} \Ga} \star \Csab(k)
\end{equation}
using the spectral measures $\Csab$ introduced in Section \ref{sec:Spectral}. An elementary computation shows that:
\begin{equation}
  \widehat{\partial^2_{ab} \Ga}(k) = - 4 \pi^2 \hPa(k) k_a k_b
\end{equation}
for all $k$ and $a,b \in \{1,2\}$. Since $\Csab$ is a measure on $\Omega$, we can write the convolution $  \widehat{\partial^2_{ab} \Ga} \star \Csab$ as:
\begin{equation}
  \widehat{\partial^2_{ab} \Ga} \star \Csab(k) = - 4 \pi^2 \int_{\Omega} \hPa(k+\omega) (k+\omega)_a (k+\omega)_b \dd \Csab(\omega),
\end{equation}
and we obtain, after summing over $a,b \in \{1,2\}$ and over $k \in \ADD$:
\begin{equation}
\MainLow =  4 \pi^2 \sum_{k \in \ADD} \sum_{a,b \in \{1,2\}}  \left( \int_{\Omega} \hPa(k+\omega) (k+\omega)_a (k+\omega)_b  \dd \Csab(\omega) - \RRab_0 \hPa(k) k_a k_b\right).  
\end{equation}
We now combine the four coefficients $\Csab$ of the matrix-valued spectral measure $\Cs$, introduce the trace measure $\tCs$ as in \eqref{eq:tCs}, and use the identity \eqref{quadraticCs} to write
\begin{multline*}
\sum_{k \in \ADD} \sum_{a,b \in \{1,2\}} \int_{\Omega}  \hPa(k + \omega) (k+\omega)_a (k+ \omega)_b \dd \Csab(\omega) \\ = 
\sum_{k \in \ADD} \int_{\Omega} \hPa(k+\omega) \left(\lambda_1(\omega) \left((k+\omega) \cdot v_1(\omega)\right)^2 + \lambda_2(\omega) \left((k+\omega) \cdot v_2(\omega)\right)^2   \right) \dd \tCs(\omega),
\end{multline*}
where $\lambda_1, \lambda_2$ are the $\omega$-dependent eigenvalues of the “trace derivative” $\Cst(\omega)$, and $v_1, v_2$ its eigenvectors, forming an orthonormal basis.
On the other hand, using the fact that $\RR_0 = \int_{\Omega} \dd \Cs$ by \eqref{GasTFrho}, and applying \eqref{quadraticCs} again,  we can write:
\begin{multline*}
\sum_{k \in \ADD} \sum_{a,b \in \{1,2\}} \RRab_0 \hPa(k) k_a k_b  = \sum_{k \in \ADD} \sum_{a,b \in \{1,2\}} \int_{\Omega} \hPa(k) k_a k_b \dd \Csab(\omega) \\ = \int_{\Omega} \sum_{k \in \ADD} \hPa(k) \left(\lambda_1(\omega) \left(k \cdot v_1(\omega)\right)^2 + \lambda_2(\omega) \left(k \cdot v_2(\omega)\right)^2   \right) \dd \tCs(\omega).
\end{multline*}
Finally, recall that $\hPa(k)  = \alpha^{-1} e^{-\frac{\pi}{\alpha}|k|^2}$ (see \eqref{eq:hPa}) and introduce, as in \eqref{def:Psia}, the function:
\begin{equation*}
\Psiav : \omega \mapsto \sum_{k \in \ADD} |(k+\omega)\cdot v|^2 e^{- \frac{\pi}{\alpha} |k+\omega|^2}.
\end{equation*}
We obtain the following expression, where $\lambda_1, \lambda_2, v_1, v_2$ are functions of $\omega$:
\begin{equation}
\label{MainLowMain}
  \MainLow = 4 \pi^2 \alpha^{-1} \left( \int_{\Omega} \left( \lambda_1 \left(\PsiavU(\omega) - \PsiavU(0)\right) +  \lambda_2 \left(\PsiavD(\omega) - \PsiavD(0)\right) \right) \dd \tCs(\omega)  \right).
\end{equation}

\paragraph{Step 4: Coercivity of $\MainLow$.} By Proposition \ref{prop:minimality}, we know that there exists $\cc >0 $ such that, for all $\alpha \in (0, \alphaL)$, for all choices of $v$ such that $|v| =1$, and all $\omega \in \Omega$: 
\begin{equation}
\Psiav(\omega) - \Psiav(0) \geq \cc |\omega|^2 e^{-\frac{\pi}{\alpha} \xstar^2}.
\end{equation}
We may thus write (recall that $\lambda_1 + \lambda_2 = 1$ almost everywhere, see Section \ref{sec:Spectral}):
\begin{equation*}
\int_{\Omega} \left( \lambda_1 \left(\PsiavU(\omega) - \PsiavU(0)\right) +  \lambda_2 \left(\PsiavD(\omega) - \PsiavD(0)\right) \right) \dd \tCs(\omega)
\geq \cc e^{-\frac{\pi}{\alpha} \xstar^2} \Tr\left(\int_{\Omega} |w|^2 \dd \Cs(\omega)\right). 
\end{equation*}
Inserting this into \eqref{MainLowMain}, and using our notation $\SM(\PP)$ (see \eqref{def:SM}), we obtain the following lower bound:
 \begin{equation}
\label{MainLowLowerbound}
 \MainLow \geq \cc e^{-\frac{\pi}{\alpha} \xstar^2} \SM(\PP), 
\end{equation}
where we discarded the $\alpha^{-1}$ factor for simplicity (here $\alpha \leq \alphaL$). For all $x \in \AD$, using Lemma \ref{lem:RRORRxIntomegacarre} we get:
\begin{equation}
\label{GGMainLow}
\int_{\R^2} |u|^2 \dd \Px(u) \leq \Cc |x|^2 \times \SM(\PP). 
\end{equation}
 Inserting \eqref{GGMainLow} into the expression \eqref{ErrorLow2} of $\ErrorLow_2$, this yields:
\begin{equation}
\label{ErrorLow2B}
\ErrorLow_2 \geq - \Cc \left( \|\PP\| \sum_{x \in \AD, |x| \leq R} |x|^2  + \sum_{x \in \AD, |x| > R} |x|^4 e^{-\pi \alpha |x|^2} \right) \SM(\PP).
\end{equation}
In conclusion, we have obtained $\Low = \MainLow + \ErrorLow_2$, with $\MainLow$ as in \eqref{MainLowMain}, bounded below in terms of $\SM(\PP)$, while $\ErrorLow_2$ is also bounded in terms of $\SM(\PP)$, see \eqref{MainLowLowerbound} and \eqref{ErrorLow2B}. 

Next, we turn to the “high” layers, for which we use a simple, rough lower bound.

\paragraph{Step 5. A lower bound on the high layers.}
Intuitively, if $R$ is large enough compared to $\alpha$, the high layers ($|x| > R$) should not play an important role due to the decaying Gaussian weight. We use a rough lower bound on their contributions:
\begin{claim}
\label{claim:highLayers}
We have
\begin{equation}
\label{eq:LBHighLayers}
\High
\geq - \Cc  \alpha \left(\sum_{x \in \AD, |x| > R} |x|^2 e^{- \frac{\pi \alpha}{2} |x|^2}\right) \times \SM(\PP). 
\end{equation}
\end{claim}
\begin{proof}
We return to the Taylor expansion used in \eqref{TaylorIntegral}, and to the expression \eqref{nab2papplied} for $\nabla^2 \Ga$, we discard the positive contribution in $ \nabla^2 \Ga(x + tu) \left(u,u\right)$ and write, for all $t \in [0,1]$ (cf. the proof of \eqref{prequantalphaL2}):
\begin{equation*}
\nabla^2 \Ga(x + tu)(u, u) \geq - 2 \pi \alpha |u|^2 e^{- \pi \alpha |x|^2 + 4 \pi \alpha |x| \|\PP\|}.
\end{equation*}
 We thus have, using \eqref{GGMainLow} in the second inequality:
\begin{multline}
\label{AnyLayerLB}
\int_{\R^2} \left(\Ga(x +u) -  \Ga(x)\right) \dd \Px(u) \geq - \Cc \alpha \left(\int_{\R^2} |u|^2 \dd \Px(u)\right) e^{- \pi \alpha |x|^2 + 4 \pi \alpha |x| \|\PP\|} \\ \geq - \Cc \alpha |x|^2 e^{- \pi \alpha |x|^2 + 4 \pi \alpha |x| \|\PP\|}  \times \SM(\PP).
\end{multline}
We now use again the fact that $e^{- \pi \alpha |x|^2 + 4 \pi \alpha |x| \|\PP\|} \leq e^{- \frac{\pi \alpha}{2} |x|^2}$. Summing over $|x| > R$ yields \eqref{eq:LBHighLayers}.
\end{proof}

\paragraph{Step 6. Combining low and high layers}
In summary, we have obtained:
\begin{equation}
\label{summaryLowHigh}
\Ea(\AD + \pP) - \Ea(\AD) = \Low + \High, \quad \Low \geq \MainLow + \ErrorLow_2,
\end{equation}
where $\MainLow$, $\ErrorLow_2$ and $\High$ are all bounded in terms of $\SM(\PP)$, see \eqref{MainLowLowerbound}, \eqref{ErrorLow2B} and \eqref{eq:LBHighLayers}. In those last two bounds, there is a lattice sum over the “high layers” $|x| > R$:
\begin{equation*}
\sum_{x \in \AD, |x| > R} |x|^4 e^{-\pi \alpha |x|^2} \text{ in \eqref{ErrorLow2B} }, \quad \sum_{x \in \AD, |x| > R}  \alpha |x|^2 e^{- \frac{\pi \alpha}{2} |x|^2} \text{ in \eqref{eq:LBHighLayers} }.
\end{equation*}
For simplicity, we bound combine them into a common bound: 
\begin{equation*}
\sum_{x \in \AD, |x| > R} |x|^4 e^{-\pi \alpha |x|^2} +\sum_{x \in \AD, |x| > R} \alpha |x|^2 e^{- \frac{\pi \alpha}{2} |x|^2}   \leq \Cc \sum_{x \in \AD, |x| > R} |x|^4 e^{- \frac{\pi \alpha}{2} |x|^2}.
\end{equation*} 
(recall that $\alpha \leq \alphaL$, and that $|x|^2 \leq \Cc |x|^4$ because even if $R = 0$, we have  $|x| \geq \xstar$). We then get:
\begin{equation}
\label{eq:EaMainLow}
\Ea(\AD + \pP) - \Ea(\AD) \geq \\ 
\cc \left(  e^{-\frac{\pi}{\alpha} \xstar^2}  - \Cc_1  \|\PP\| \sum_{x \in \AD, |x| \leq R} |x|^2 - \Cc_1 \sum_{x \in \AD, |x| > R} |x|^4 e^{- \frac{\pi \alpha}{2} |x|^2} \right) \times \SM(\PP),
\end{equation}
for some constant $\Cc_1$ independent on $\alpha \in (0, \alphaL)$. The positive term in the parenthesis comes from $\MainLow$, see \eqref{MainLowLowerbound}, and the two negative ones from the lower bounds on $\ErrorLow_2$ and $\High$.

\paragraph{Step 7. Choosing $R$.} We now take $R$ of the form $\frac{M}{\alpha}$, with $M$ “large”. By comparing the lattice sum to a Gaussian integral, we see that there exists $\Cc_0 > 0$ such that, for all $\alpha \leq \alphaL$ and all $R \geq \frac{\Cc_0}{\alpha}$:
\begin{equation*}
\sum_{x \in \AD, |x| > R} |x|^4 e^{- \frac{\pi \alpha}{2} |x|^2} \leq \Cc R^4 e^{- \frac{\pi \alpha}{2} R^2}.
\end{equation*}
We thus see that, uniformly for $M \geq \Cc_0$ and for $\alpha \in (0, \alphaL)$:
\begin{equation*}
\sum_{x \in \AD, |x| > \frac{M}{\alpha}} |x|^4 e^{- \frac{\pi \alpha}{2} |x|^2} \leq \Cc \alpha^{-4} M^4 e^{- \frac{\pi M^2}{2 \alpha}} \leq \Cc' M^4 e^{- \frac{\pi M^2}{4 \alpha}}.
\end{equation*}
Choosing $M$ large enough, we can then ensure, for all $\alpha \leq \alphaL$:
\begin{equation}
\label{HighSmall}
\Cc_1  \sum_{x \in \AD, |x| > \frac{M}{\alpha}} |x|^4 e^{- \frac{\pi \alpha}{2} |x|^2} \leq \Cc_1 \Cc' M^4 e^{- \frac{\pi M^2}{4 \alpha}} \leq \frac{1}{4} e^{-\frac{\pi}{\alpha} \xstar^2}.
\end{equation}
Moreover, we our choice $R = \frac{M}{\alpha}$, we have the upper bound:
\begin{equation}
\label{transfertootherterm}
\|\PP\| \sum_{x \in \AD, |x| \leq \frac{M}{\alpha}} |x|^2 \leq \Cc{''} \|\PP\| M^4 \alpha^{-4}.
\end{equation}
Our condition $\|\PP\| \leq \cc e^{-\frac{2\pi}{\alpha} \xstar^2}$ with $\cc$ small enough ensures that, for $\alpha \in (0, \alphaL)$:
\begin{equation}
\label{ErrorLowsmall}
\Cc_1  \|\PP\|\sum_{x \in \AD, |x| \leq R} |x|^2 \leq \cc \Cc_1 \Cc{''} M^4 \alpha^{-4} e^{-\frac{2\pi}{\alpha} \xstar^2} \leq \frac{1}{4}  e^{-\frac{\pi}{\alpha} \xstar^2}.
\end{equation}
Inserting the controls \eqref{HighSmall} and \eqref{ErrorLowsmall} into \eqref{eq:EaMainLow}, we obtain \eqref{LowerboundAlphaSmall}

\vspace{0.2cm}

We conclude the proof of Lemma \ref{lem:smallalpha} by proving \eqref{quantitativealphaSMALL_2}. We now only assume that $\|\PP\|$ satisfies \eqref{eq:pp100}.
\paragraph{Step 8. Proof of \eqref{quantitativealphaSMALL_2}}
We choose $R = 0$, so there are only “high” layers, and we use \eqref{eq:LBHighLayers}, which gives:
\begin{equation*}
\Ea(\AD + \pP) - \Ea(\AD)  \geq - \Cc \alpha \left(\sum_{x \in \AD} |x|^2 e^{- \frac{\pi \alpha}{2} |x|^2} \right) \times \SM(\PP).
\end{equation*}
For $\alpha \in (0, \alphaL)$, the sum is bounded by $\Cc \alpha^{-2}$ (this can be seen e.g. by applying Poisson's summation formula, or by comparing to a Gaussian integral), which yields \eqref{quantitativealphaSMALL_2}.

\subsection{Local optimality for c.m.s.d. functions: proof of Theorem \ref{theo:Family}}
Let $f$ be a c.m.s.d function and let $\muf$ be the associated measure on $[0, + \infty)$, such that
\begin{equation*}
f = \int_{0}^{+\infty} \Pa \dd \muf(\alpha).
\end{equation*}
In this section, we find conditions on $\muf$ ensuring the existence of $\epsilon > 0$ such that for all periodic perturbations $\PP$:
\begin{equation*}
\text{If } \|\PP\| \leq \epsilon \text{ then } \Ef(\AD + \PP) \geq \Ef(\AD).
\end{equation*}
Moreover, we observe that the choice of $\epsilon$ can be done uniformly on certain families of functions $f$, as stated in Theorem \ref{theo:Family}. The starting point is to write:
\begin{equation*}
\Ef(\AD + \PP) - \Ef(\AD) = \int_{0}^{+\infty} \left( \Ea(\AD + \PP) - \Ea(\AD) \right) \dd \muf(\alpha),
\end{equation*}
which should be quickly justified: if $f$ is a c.m.s.d function such that $x \mapsto f(|x|)$ is integrable at infinity (otherwise $\Ef(\AD)$ and $\Ef(\AD + \PP)$  are infinite, see Remark \ref{rem:integrability}) then the $\liminf$'s defining $\Ef(\AD), \Ef(\AD + \PP)$ in \eqref{def:Ef} are limits, and we simply write:
\begin{multline*}
\Ef(\AD + \PP) = \lim_{r \to \infty} \frac{1}{|(\AD + \PP) \cap \BR|} \sum_{x, y \in  (\AD + \PP) \cap \BR, x \neq y} \int_{0}^{+\infty} \Pa(|x-y|) \dd \muf(\alpha)\\ = \int_{0}^{+\infty} \left(\lim_{r \to \infty} \frac{1}{|(\AD + \PP) \cap \BR|} \sum_{x, y \in (\AD + \PP) \cap \BR, x \neq y} \Pa(|x-y|) \right) \dd \muf(\alpha).
\end{multline*}

\subsubsection*{Step 1. Selecting $\epsilon$ for the large $\alpha$'s}
Let $\alphaL$ be the threshold given by Lemma \ref{lem:largealpha}. We first focus on the contributions coming from $\alpha \geq \alphaL$.
\begin{claim}[Large $\alpha$'s]
Let $\alpha_1 \geq \alphaL$ to be chosen later, and assume that $\|\PP\| \leq \frac{1}{\alpha_1}$. We have:
\begin{multline}
\label{eq:LargeAlphaSelection}
\int_{\alphaL}^{+ \infty} \left(\Ea(\AD + \PP) - \Ea(\AD) \right) \dd \muf(\alpha) \\ \geq \left(\cc \int_{\alphaL}^{+ \infty}  e^{-\pi \alpha \xstar^2} \dd \muf(\alpha) - \Cc \int_{\alpha_1}^{+ \infty} e^{- \frac{\pi \alpha}{2} \xstar^2} \dd \muf(\alpha) \right) \FS(\PP).
\end{multline}
\end{claim}
\begin{proof}
We split the integral into two parts:
\begin{itemize}
   \item  For $\alpha$ between $\alphaL$ and $\alpha_1$, since we assume that $\|\PP\| \leq \frac{1}{\alpha_1} \leq \frac{1}{\alpha}$ we can use our quantitative local minimality statement \eqref{quantitativealphaL1} and integrate it over $\alpha$, which yields:
  \begin{equation*}
\int_{\alphaL}^{\alpha_1} \left(\Ea(\AD + \PP) - \Ea(\AD) \right) \dd \muf(\alpha) \geq \cc \int_{\alphaL}^{\alpha_1} e^{-\pi \alpha \xstar^2} \dd \muf(\alpha) \times \FS(\PP).
  \end{equation*}
\item If $\alpha \geq \alpha_1$, we use instead the rough lower bound \eqref{quantitativealphaL2}, which gives:
\begin{equation*}
\int_{\alpha_1}^{+ \infty}  \left(\Ea(\AD + \PP) - \Ea(\AD) \right) \dd \muf(\alpha) \geq - \Cc \int_{\alpha_1}^{+ \infty} e^{- \frac{\pi \alpha}{2} \xstar^2} \dd \muf(\alpha) \times \FS(\PP).
\end{equation*}
\end{itemize}
We also write, for simplicity: 
\begin{multline*}
\int_{\alphaL}^{\alpha_1}  e^{-\pi \alpha \xstar^2} \dd \muf(\alpha) = \int_{\alphaL}^{+ \infty}  e^{-\pi \alpha \xstar^2} \dd \muf(\alpha) - \int_{\alpha_1}^{+ \infty}  e^{-\pi \alpha \xstar^2} \dd \muf(\alpha) \\  \geq \int_{\alphaL}^{+ \infty}  e^{-\pi \alpha \xstar^2} \dd \muf(\alpha) - \int_{\alpha_1}^{+ \infty}  e^{- \frac{\pi \alpha}{2} \xstar^2} \dd \muf(\alpha).
\end{multline*}
We obtain \eqref{eq:LargeAlphaSelection}. 
\end{proof}
Next, we focus on the contributions to $\EE_f$ coming from $\alpha \leq \alphaL$.

\subsubsection*{Step 2. Selecting $\epsilon$ for the small $\alpha$'s}
\begin{claim}
\label{claim:epsilon_smallapha}
Let $0 < \alpha_0 \leq \alphaL$ to be chosen later, and assume that $\|\PP\| \leq \cc_1 e^{-\frac{2 \pi}{\alpha_0} \xstar^2}$ with $\cc_1$ the constant appearing in Lemma \ref{lem:smallalpha}. We have:
\begin{equation}
\label{eq:epsilon_smallapha}
\int_{0}^{\alphaL} \left(\Ea(\AD + \PP) - \Ea(\AD) \right) \dd \muf(\alpha) \geq \left(\cc_1 \int_{0}^{\alphaL} e^{-\frac{\pi}{\alpha} \xstar^2} \dd \muf(\alpha) - \Cc \int_{0}^{\alpha_0}  e^{-\frac{\pi}{2\alpha} \xstar^2} \dd \muf(\alpha) \right) \times \SM(\PP).
\end{equation}
\end{claim}
\begin{proof}
We split again the integral into two parts:
\begin{itemize}
  \item If $\alpha$ is between $\alpha_0$ and $\alphaL$, since we assume that $\|\PP\| \leq \cc_1 e^{-\frac{2\pi}{\alpha_0} \xstar^2} \leq \cc_1 e^{-\frac{2\pi}{\alpha} \xstar^2}$, we can use our quantitative local minimality bound \eqref{LowerboundAlphaSmall}, which yields:
\begin{equation*}
\int_{\alpha_0}^{\alphaL} \left(\Ea(\AD + \PP) - \Ea(\AD) \right) \dd \muf(\alpha) \geq \cc_1 \int_{\alpha_0}^{\alphaL} e^{-\frac{\pi}{\alpha} \xstar^2} \dd \muf(\alpha) \times \FS(\PP).
\end{equation*}
\item If $\alpha$ is smaller than $\alpha_0$, we use instead the rough lower bound \eqref{quantitativealphaSMALL_2}, which gives:
\begin{equation*}
\int_{0}^{\alpha_0} \left(\Ea(\AD + \PP) - \Ea(\AD) \right) \dd \muf(\alpha) \geq - \Cc \int_{0}^{\alpha_0}  \alpha^{-1} \dd \muf(\alpha) \times \FS(\PP).
\end{equation*}
\end{itemize}
Again, for simplicity, we write:
\begin{multline*}
\int_{\alpha_0}^{\alphaL} e^{-\frac{\pi}{\alpha} \xstar^2} \dd \muf(\alpha) = \int_{0}^{\alphaL} e^{-\frac{\pi}{\alpha} \xstar^2} \dd \muf(\alpha) - \int_{0}^{\alpha_0} e^{-\frac{\pi}{\alpha} \xstar^2} \dd \muf(\alpha) \\ \geq \int_{0}^{\alphaL} e^{-\frac{\pi}{\alpha} \xstar^2} \dd \muf(\alpha) - \int_{0}^{\alpha_0} e^{-\frac{\pi}{2 \alpha} \xstar^2} \dd \muf(\alpha) \geq  \int_{0}^{\alphaL} e^{-\frac{\pi}{\alpha} \xstar^2} \dd \muf(\alpha) - \int_{0}^{\alpha_0} \alpha^{-1} \dd \muf(\alpha),
\end{multline*}
and we obtain \eqref{eq:epsilon_smallapha}.
\end{proof}

\subsubsection*{Step 3. Conclusion}
In summary, for all $\alpha_0 \leq \alphaL \leq \alpha_1$, assuming that
\begin{equation}
\label{epsilonpourPP}
\|\PP\| \leq \min\left(\frac{1}{\alpha_1}, \cc_1 e^{-\frac{2\pi}{\alpha_0} \xstar^2}\right),
\end{equation}
we obtain the following lower bound on $\Ef(\AD + \PP) - \Ef(\AD)$ (recall that $\SM(\PP)$ controls $\FS(\PP)$ by \eqref{G0Gxrho}):
\begin{multline}
\label{LBFamily}
\Ef(\AD + \PP) - \Ef(\AD) \geq \Bigg(\cc \left(\int_{\alphaL}^{+ \infty}  e^{-\pi \alpha \xstar^2} \dd \muf(\alpha) + \int_{0}^{\alphaL} e^{-\frac{\pi}{\alpha} \xstar^2} \dd \muf(\alpha)\right) \\ - \Cc \left(\int_{\alpha_1}^{+ \infty} e^{-\frac{\pi\alpha}{2} \xstar^2} \dd \muf(\alpha) + \int_{0}^{\alpha_0}  \alpha^{-1} \dd \muf(\alpha)\right) \Bigg) \times \FS(\PP).
\end{multline}
We make two observations:
\begin{enumerate}
  \item By definition, for all $r > 0$, the integral $\int_{0}^{+ \infty}  e^{-\pi \alpha r^2} \dd \muf(\alpha) = f(r)$ is finite, so in particular $\alpha \mapsto e^{-\frac{\pi \alpha}{2} \xstar^2}$ is $\muf$-integrable at infinity.
  \item The condition $\Ef(\AD) < + \infty$ is equivalent to $r \mapsto r f(r)$ being integrable at infinity, see Remark \ref{rem:integrability}. A direct computation shows that:
  \begin{equation}
  \label{TwoInteg}
\int_{1}^{+ \infty} rf(r) \dd r=  \int_{0}^{+ \infty}  \int_{1}^{+ \infty} r e^{-\pi \alpha r^2} \dd r \ \dd \muf(\alpha) = \int_{0}^{+ \infty} \alpha^{-1} \int_{\sqrt{\alpha}}^{+ \infty} e^{- \pi u^2} \dd u \ \dd \muf(\alpha),
  \end{equation}
  and thus this finiteness condition implies that $\alpha \mapsto \alpha^{-1}$ must be $\muf$-integrable near $0$.
\end{enumerate}
As a consequence, by choosing $\alpha_1$ large enough and $\alpha_0$ small enough (depending on $f$), we can ensure that the parenthesis in \eqref{LBFamily} is positive. Taking $\epsilon$ as the right-hand side of \eqref{epsilonpourPP}, and imposing $\|\PP\| \leq \epsilon$, we deduce that $\AD$ is locally optimal for $\Ef$. This proves the first statement in Theorem \ref{theo:Family}.

Moreover, we see that $\alpha_0, \alpha_1$ can be chosen uniformly over families $\FF$ of c.m.s.d functions $f$ such that:
\begin{equation*}
\lim_{\alpha_0 \to 0, \alpha_1 \to \infty} \frac{\int_{\alpha_1}^{+ \infty} e^{- \frac{\pi \alpha}{2} \xstar^2} \dd \muf(\alpha) + \int_{0}^{\alpha_0} \alpha^{-1} \dd \muf(\alpha)}{\int_{\alphaL}^{+ \infty}  e^{-\pi \alpha \xstar^2} \dd \muf(\alpha) + \int_{0}^{\alphaL} e^{-\frac{\pi}{\alpha} \xstar^2} \dd \muf(\alpha)} = 0 \text{ uniformly for $f \in \FF$.}
\end{equation*}
Here $\alphaL \geq 1$ is the threshold from Lemma \ref{lem:largealpha}, for which we never searched an explicit value. For simplicity, we replace it by $1$ in the integral bounds, which is valid because:
\begin{equation*}
 \int_{1}^{+ \infty}  e^{-\pi \alpha \xstar^2} \dd \muf(\alpha) + \int_{0}^{1} e^{-\frac{\pi}{\alpha} \xstar^2} \dd \muf(\alpha) \leq  \Cc \left(\int_{\alphaL}^{+ \infty}  e^{-\pi \alpha \xstar^2} \dd \muf(\alpha) + \int_{0}^{\alphaL} \alpha^{-1} \dd \muf(\alpha)  \right) ,
\end{equation*}
indeed, $e^{-\pi \alpha \xstar^2}$, $e^{-\frac{\pi}{\alpha} \xstar^2}$ are comparable for $\alpha \in [1, \alphaL]$. This concludes the proof of Theorem \ref{theo:Family}.

\clearpage

\appendix
\section{Reduction to periodic perturbations: proof of Lemma \ref{lem:Station}}
\label{sec:proof_periodic}
Let $f$ be a c.m.s.d function - in particular, $f$ is smooth and decreasing. Without loss of generality, we can assume that $x \mapsto f(|x|)$ (or equivalently $r \mapsto r f(r)$) is integrable at $\infty$, otherwise both $\Ef(\AD)$ and $\Ef(\AD + \PP)$ are infinite and there is nothing to prove.

For $N \geq 1$, let $\La_N$ be the sub-lattice of $\AD$ generated by the vectors $2N \sigma$ and $2N \tau$, and let $\LAN$ be a fundamental domain (here a parallelogram) for $\La_N$ given by $\LAN := \left\lbrace s N \sigma + t N \tau, \ (s,t) \in [-1,1]^2 \right\rbrace$. 

The sequence of shapes $(\LAN)_{N \geq 1}$ converges to $\R^2$ in the sense of Van Hove, see \cite[Sec.~3.2.1]{friedli2017statistical}. It is folklore that the following limit exists (see e.g. \cite[Sec. 2.4]{blanc2015crystallization}):
\begin{equation*}
\lim_{r \to \infty} \frac{1}{|(\AD + \PP) \cap \BR|} \sum_{x, y \in (\AD + \PP) \cap \BR, x \neq y} f(|x-y|),
\end{equation*}
cf. \eqref{def:Ef}, and coincides with the same limit taken along the sequence $(\LAN)_{N \geq 1}$, thus we have:
\begin{equation}
\label{limLN}
\Ef(\AD + \pP) = \lim_{N \to \infty} \frac{1}{|(\AD + \pP) \cap \LAN|} \sum_{x, y \in (\AD + \pP)  \cap \LAN, \ x \neq y} f(|x-y|).
\end{equation}
In particular, we can take $N$ large enough such that:
\begin{equation}
\label{VH1} 
\frac{1}{|(\AD + \pP) \cap \Lambda_{N}|} \sum_{x, y \in (\AD + \pP)  \cap \Lambda_{N}, x \neq y} f(|x-y|) \leq \Ef(\AD + \pP) + \delta.
\end{equation}
We restrict $\PP$ to the parallelogram $\Lambda_{N}$, and define $\PP_{\per}$ as the $\La_{N}$-periodic extension of $\PP$ to $\AD$. Of course, we have $\|\PP_{\per}\| \leq \|\PP\|$. It remains to prove that $\Ef(\AD + \PP_{\per})$ is not much larger than $\Ef(\AD + \PP)$.

Consider the sub-sequence of shapes $(\Lambda_{kN})_{k \geq 1}$. For the same reason as above, we have:
\begin{equation}
\label{VH2}
\Ef(\AD + \PP_{\per}) = \lim_{k \to \infty}  \frac{1}{|(\AD + \PPer) \cap \Lambda_{kN}|} \sum_{x, y \in (\AD + \PPer)  \cap \Lambda_{kN}, \ x \neq y} f(|x-y|).
\end{equation}
Notice that $\Lambda_{kN}$ is the disjoint union of $k^2$ disjoint copies of $\LAN$. Since $\PPer$ is periodic by construction, and coincides with $\PP$ on $\Lambda_{N}$, we have:
\begin{equation*}
|(\AD + \PPer) \cap \Lambda_{kN}| = k^2 |(\AD + \PP) \cap \Lambda_{N}|.
\end{equation*}
Moreover, the energy within $\Lambda_{kN}$ is given by $k^2$ times the energy within $\Lambda_{N}$, plus the interactions between each copy of $\LAN$ and the rest of the configuration. Since $f$ is decreasing and $x \mapsto f(|x|)$ is integrable at infinity, we have, for $x \in \Lambda_{N}$
\begin{equation*}
\sum_{y \in (\AD + \PP_{\per}) \setminus \Lambda_{N}} |f(x-y)| = o_{\dist(x, \partial \Lambda_{N}) \to \infty}(1),
\end{equation*}
and thus we can bound the interaction between each copy and the rest of the configuration by:
\begin{equation*}
\sum_{x \in (\AD + \PPer) \cap \Lambda_{N}, y \in (\AD + \PP_{\per}) \setminus \Lambda_{N}} \left|f(|x-y|)\right| = o\left(|\Lambda_{N}|\right),
\end{equation*}
using the fact that boundary contributions are negligible with respect to the volume. In particular, choosing $N$ large enough, we have:
\begin{equation*}
\frac{1}{|\Lambda_{N}|} \sum_{x \in (\AD + \PP_{\per})  \cap \Lambda_{N}, y \in (\AD + \PP_{\per}) \setminus \Lambda_{N}} |f(x-y)| \leq \delta,
\end{equation*}
which in particular implies for all $k \geq 1$:
\begin{equation*}
 \sum_{x, y \in (\AD + \PPer)  \cap \Lambda_{kN}, \ x \neq y} f(|x-y|) \leq k^2 \times  \sum_{x, y \in (\AD + \PP)  \cap \Lambda_{N}, \ x \neq y} f(|x-y|) +  k^2 \delta.
\end{equation*}
Dividing by $|(\AD + \PPer) \cap \Lambda_{kN}| = k^2 |(\AD + \PP) \cap \Lambda_{N}|$, inserting \eqref{VH1}, and sending $k \to \infty$ (see \eqref{VH2}) yields the result.

\clearpage

\section{Proof of the “geometric” Lemma \ref{lem:geomwk}}
\label{sec:proof_of_lemma_ref_lem_geomwk}
Without loss of generality, we assume that $S$ has sidelength $r = 1$.  For $k,v$ fixed, we want to prove: 
\begin{equation*}
\sum_{s \in S} w_s(k) |s \cdot v|^2 \geq \frac{1}{4} \left(\sum_{s \in S} w_s(k)\right) |v|^2, \quad \text{ with } w_s(k) := 2 \left(1 - \cos\left(2 \pi k \cdot s\right) \right).
\end{equation*}
This can be rewritten in “linear algebraic” terms as:
\begin{equation}
\label{eq:average2designM}
\Mm v \cdot v \geq \frac{1}{4} |v|^2, \quad \text{ where } \Mm := \frac{1}{\sum_{s \in S} w_s(u)} \sum_{s \in S} w_s(k) s s^{T}.
\end{equation}
The matrix $\Mm$ is positive symmetric as a mixture of orthogonal projections. We can write it down explicitly using the fact that the vertices of $S$ are given by:
\begin{equation*}
s_1 := (1,0), \quad s_2 := \left(\frac{1}{2}, \frac{\sqrt{3}}{2}\right), \quad s_3 = s_2 - s_1 = \left(-\hal, \frac{\sqrt{3}}{2}\right), 
\end{equation*}
together with their opposites, which yield the same contributions.
We obtain:
\begin{equation*}
\Mm = \frac{1}{w_1 + w_2 + w_3} \begin{pmatrix}
w_1+\tfrac{w_2+w_3}{4} & \tfrac{\sqrt3}{4}(w_2-w_3)\\[1mm]
\tfrac{\sqrt3}{4}(w_2-w_3) & \tfrac{3}{4}(w_2+w_3)
\end{pmatrix}
\end{equation*}
where we write $w_1, w_2, w_3$ instead of $w_{s_1}(k), w_{s_2}(k), w_{s_3}(k)$. Clearly, $\Tr \Mm = 1$, and $\det \Mm$ is given by
\begin{equation}
\label{detM}
\det \Mm = \frac{3}{4} \frac{w_1 w_2 + w_2w_3 + w_1w_3}{(w_1 + w_2 + w_3)^2}.
\end{equation}
An elementary computation shows that the smallest eigenvalue $\lambda_{\min}$ of $\Mm$ is then equal to 
\begin{equation*}
\lambda_{\min} = \frac{1 - \sqrt{1 - 4 \det \Mm}}{2}.
\end{equation*}
Proving that $\lambda_{\min} \geq \frac{1}{4}$ will imply \eqref{eq:average2designM}. By our expression \eqref{detM} for $\det \Mm$, this is equivalent to having:
\begin{equation}
\frac{w_1 w_2 + w_2w_3 + w_1w_3}{(w_1 + w_2 + w_3)^2} \geq \frac{1}{4}.
\end{equation}
Such an inequality is of course false if $w_1, w_2, w_3$ are arbitrary positive numbers, because one of them could be much larger than the other two, making the ratio very small. We now claim that since $s_3 = s_2 - s_1$, there cannot be one term among $w_1, w_2, w_3$ which is much larger than the other two. 
\begin{claim}
We have:
\begin{equation}
\label{eq:w12314}
w_1 w_2 + w_2w_3 + w_1w_3 \geq \frac{1}{4} (w_1 + w_2 + w_3)^2.
\end{equation}
\end{claim}
\begin{proof}
For $j \in \{1, 2, 3\}$ let $\theta_j := 2\pi s_j \cdot k$ and let $c_j := \cos \theta_j$. We have by definition:
\begin{equation*}
w_j = 2\left(1 - c_j\right) = 4 \sin^2\left( \frac{\theta_j}{2} \right).
\end{equation*}
The inequality \eqref{eq:w12314} can then be re-written, after some algebra, as:
\begin{equation}
\label{ineq:c123}
4 \left( c_1 c_2 + c_2 c_3 + c_3 c_1  \right) - \left(c_1 + c_2 + c_3\right)^2  - 2 \left(c_1 + c_2 + c_3\right) + 3 \geq 0.
\end{equation}
Introduce the auxiliary variable $t := \sin \left(\theta_1\right) \sin \left(\theta_2\right)$. Since $s_3 = s_2-s_1$ as vectors, we have:
\begin{equation*}
c_3 = \cos\left(\theta_3\right) = \cos\left( 2\pi s_3 \cdot k \right) = \cos\left( 2\pi (s_2-s_1)\cdot k \right) = \cos(\theta_2 - \theta_1) = c_1 c_2 + \sin(\theta_1) \sin(\theta_2) = c_1 c_2 + t.
\end{equation*}
Injecting this, and noticing that $t^2 = (1-c_1^2)(1-c_2^2)$, we can re-write the left-hand side of \eqref{ineq:c123} as:
\begin{multline}
\label{c123_2}
4 \left( c_1 c_2 + c_2 c_3 + c_3 c_1  \right) - \left(c_1 + c_2 + c_3\right)^2  - 2 \left(c_1 + c_2 + c_3\right) + 3 \\
= 4 \left(1 - c_1\right)\left(1-c_2\right) - \left(t + (1-c_1)(1-c_2)  \right)^2.
\end{multline}
Finally, some elementary trigonometry yields:
\begin{equation*}
\left(1 - c_1\right)\left(1-c_2\right) = 4 \sin^2\left(\frac{\theta_1}{2} \right) \sin^2\left(\frac{\theta_2}{2} \right), \quad t + (1-c_1)(1-c_2) = 4 \sin\left(\frac{\theta_1}{2} \right) \sin\left(\frac{\theta_2}{2} \right)\cos\left(\frac{\theta_1 - \theta_2}{2} \right).
\end{equation*}
The inequality \eqref{ineq:c123} expresses that the right-hand side of \eqref{c123_2} is non-negative, and is thus equivalent to:
\begin{equation*}
16 \sin^2\left(\frac{\theta_1}{2} \right) \sin^2\left(\frac{\theta_2}{2} \right) -  16 \sin^2\left(\frac{\theta_1}{2} \right) \sin^2\left(\frac{\theta_2}{2} \right)\cos^2\left(\frac{\theta_1 - \theta_2}{2} \right) \geq 0,
\end{equation*}
which is true, and concludes the proof of the claim.
\end{proof}
Now, the inequality \eqref{eq:w12314} being satisfied, we get \eqref{eq:average2designM}, which yields the result.

\section{Proof of Proposition \ref{prop:minimality}}
\label{sec:proof_minimality}
Recall that $v$ is a fixed unit vector and that $\Psiav$ is defined on the fundamental domain $\HH$ of $\AD$ by:
\begin{equation}
\label{PsiaOri}
\Psiav : u \mapsto \sum_{x \in \AD} |(x+u) \cdot v|^2 e^{- \frac{\pi}{\alpha} |x+u|^2}.
\end{equation}
We want to show that $\Psia$ is minimal at $0$, with a lower bound of the form $\Psiav(u) - \Psiav(0) \geq \cc e^{- \frac{\pi}{\alpha} \xstar^2}  |u|^2$.
\begin{figure}[h!]
\begin{center}
\begin{tikzpicture}[scale=1.3]
  \pgfmathsetmacro{\xstar}{sqrt(2/sqrt(3))}
  \pgfmathsetmacro{\R}{\xstar/sqrt(3)}            
  \pgfmathsetmacro{\hx}{0.5*\xstar}
  \pgfmathsetmacro{\hy}{0.5*\xstar/sqrt(3)}

  \coordinate (A1) at ( 0,  \R);
  \coordinate (A2) at (-\hx, \hy);
  \coordinate (A3) at (-\hx,-\hy);
  \coordinate (A4) at ( 0, -\R);
  \coordinate (A5) at ( \hx,-\hy);
  \coordinate (A6) at ( \hx, \hy);

  \draw[fill=gray!20,thick] (A1)--(A2)--(A3)--(A4)--(A5)--(A6)--cycle;
  \node at (0,0.2) {$\HH$};

  \draw[-,very thick] (A1) -- (A2)
       node[midway,above=1pt] {$\tfrac{r_\star}{\sqrt3}$};

  \foreach \px/\py in {
      1/0,  0.5/0.866025403784,
     -0.5/0.866025403784, -1/0,
     -0.5/-0.866025403784, 0.5/-0.866025403784}
    { \fill (\px*\xstar,\py*\xstar) circle(0.02); }

  \fill (0,0) circle(0.02) node[below left=-2pt] {$0$};
\end{tikzpicture}
\caption{A close-up of the fundamental hexagon $\HH$, surrounded by the first “shell” of the lattice.}
\end{center} 
\end{figure}
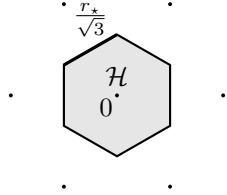
\newcommand{\ycrit}{y_{\mathrm{crit}}}
\newcommand{\tmax}{t_{\max}}
\newcommand{\Kmin}{K_{\min}}
\newcommand{\ymax}{y_{\max}}
\newcommand{\muc}{\bar{\mu}}

We will use the following numerical constants: 
 \begin{equation*}
\xstar = \sqrt{\frac{2}{\sqrt{3}}} \approx 1.0746, \quad \alphac := 0.552.
 \end{equation*}
We distinguish between two ranges for $\alpha$: “large” (between $\alphac$ and $\alphaL$) and “small” (between $0$ and $\alphac$).
\subsection{First case: \texorpdfstring{$\alpha$ large}{alpha large}.} 
\begin{lemma}
Let $\alphaL$ be fixed. There exists $\cc > 0$ such that for all $\alpha \in (\alphac, \alphaL)$, for all unit vector $v$, and all $u$ in $\HH$, we have: 
\begin{equation}
\label{210large}
  \Psiav(u) - \Psiav(0) \geq \cc e^{- \frac{\pi}{\alpha} \xstar^2}  |u|^2.
\end{equation}
\end{lemma}
\begin{proof}
Assume that $\alpha \in (\alphac, \alphaL)$ and fix a unit vector $v$.

\paragraph{Step 1. Moving to Fourier side.} Taking the Fourier transform of $x \mapsto |(x+u) \cdot v|^2 e^{- \frac{\pi}{\alpha} |x+u|^2}$ in $\R^2$, for $u \in \R^2$ fixed and $|v| = 1$ yields:
\begin{equation*}
k \mapsto \frac{\alpha^2}{2 \pi} \left( 1 - 2 \pi \alpha |k \cdot v|^2  \right) e^{-\pi \alpha |k|^2} e^{2i \pi u \cdot k},
\end{equation*}
and thus, applying Poisson's summation formula, for all  $u \in \HH$ we have the identity:
\begin{equation}
\Psia(u)  = \frac{\alpha^2}{2 \pi} \left(1  + \sum_{k \in \ADD, k \neq 0} \left(1 - 2 \pi \alpha |k \cdot v|^2  \right) e^{-\pi \alpha |k|^2} \cos(2 \pi u \cdot k) \right),
\end{equation}
and thus in particular, we can write the difference $\Psia(u) - \Psia(0)$ as:
\begin{equation}
\label{PoissonPsia}
\Psia(u) - \Psia(0) =  \frac{\alpha^2}{2\pi} \sum_{k \in \ADD, k \neq 0} \left(2\pi \alpha |k \cdot v|^2  - 1 \right) e^{-\pi \alpha |k|^2} \left(1 - \cos(2 \pi u \cdot k) \right).
\end{equation}

\paragraph{Step 2. Using the periodic $2$-design property.} Fix a shell $S$ of radius $r > 0$ within $\ADD$ and write:
\begin{equation}
\label{prepialphahal}
\sum_{s \in S} \left(2\pi \alpha |s \cdot v|^2  - 1 \right) e^{-\pi \alpha |s|^2} \left(1 - \cos(2 \pi u \cdot s) \right) = \left(\pi \alpha \sum_{s \in S} w_s(u) |s \cdot v|^2 - \frac{1}{2} \sum_{s \in S} w_s(u) \right) e^{-\pi \alpha r^2},
\end{equation}
with non-negative weights $w_s(u) = 2\left(1 - \cos(2 \pi u \cdot s) \right)$ as in \eqref{def:wsu}. Applying Lemma \ref{lem:geomwk} we get:\begin{equation}
\label{pialphahal}
\pi \alpha \sum_{s \in S} w_s(u) |s \cdot v|^2 - \frac{1}{2} \sum_{s \in S} w_s(u) \geq \left(\frac{1}{4} \pi \alpha r^2 - \hal\right) \sum_{s \in S} w_s(u) = \frac{1}{4} \pi r^2 \left(\alpha - \frac{2}{\pi r^2} \right) \sum_{s \in S} w_s(u) .
\end{equation}
Since $r \geq \xstar$, since we assume $\alpha \geq \alphac$, and since the following numerical inequality is true:
\begin{equation}
\label{NUM_ALPHACALPHA}
\alphac > \frac{2}{\pi \xstar^2} = \frac{\sqrt{3}}{\pi} \quad (0.552 > 0.55132..)
\end{equation}
the right-hand side of \eqref{pialphahal} is non-negative, thus so is \eqref{prepialphahal}.

\paragraph{Step 3. Quadratic lower bound.} All the shells have a positive contribution, we focus on the first one.
\begin{claim}
For all $u \in \HH$, we have, $\Ss$ being the first shell of the lattice $\ADD$:
\begin{equation}
\label{wsufirstshell}
\sum_{s \in \Ss} w_s(u) \geq 48 \xstar^2 |u|^2.
\end{equation}
\end{claim}
\begin{proof}
By the choice of the fundamental cell $\HH$, we have $|u \cdot s| \leq \hal$ for all $s \in \Ss$ and $u \in \HH$. Using the lower bound $1-\cos x \geq \frac{2}{\pi^2} x^2$ valid for $|x| \leq \pi$, we get:
\begin{equation*}
w_s(u)=2\left(1-\cos(2\pi u \cdot s)\right)\geq 16 |u \cdot s|^2.
\end{equation*}
Summing over the six vertices of $\Ss$ and using the $2$-design property \eqref{eq:2design}, we get \eqref{wsufirstshell}.
\end{proof}
Combining \eqref{prepialphahal}, \eqref{pialphahal} and \eqref{wsufirstshell} we obtain:
\begin{equation}
  \sum_{s \in \Ss} \left(2\pi \alpha |s \cdot v|^2  - 1 \right) e^{-\pi \alpha |s|^2} \left(1 - \cos(2 \pi u \cdot s) \right) \geq 12 \pi \xstar^4  \left(\alpha - \frac{2}{\pi \xstar^2} \right) |u|^2.
\end{equation}
Returning to \eqref{PoissonPsia} and keeping only the first shell as a lower bound, we thus get:
\begin{equation*}
\Psiav(u) - \Psiav(0) \geq 6 \alpha^2 \xstar^4 \left( \alpha - \frac{2}{\pi \xstar^2} \right) e^{-\pi \alpha \xstar^2} |u|^2. 
\end{equation*} 
For $\alpha \geq \alphac > \frac{2}{\pi \xstar^2}$, we can write, for some small universal constant $\cc$
\begin{equation*}
	6 \xstar^4 \left( \alpha - \frac{2}{\pi \xstar^2} \right) \geq \cc \alpha.
\end{equation*}
We have thus obtained
\begin{equation*}
\Psiav(u) - \Psiav(0) \geq \cc \alpha^3 e^{- \pi \alpha \xstar^2} |u|^2.
\end{equation*}
It is more convenient for us to re-write this as in \eqref{210large}. This is valid because, for $\alpha$ between $\alphac$ and $\alphaL$:
\begin{equation*}
\alpha^3 e^{- \pi \alpha \xstar^2} \geq \cc e^{- \frac{\pi}{\alpha} \xstar^2}.
\end{equation*}
\end{proof}

\subsection{Second case: \texorpdfstring{$\alpha$}{alpha} small.}
Returning to the original expression \eqref{PsiaOri} for $\Psiav$, we write for $u \in \HH$:
\begin{equation}
\label{Psiaseparation}
\Psiav(u) = |u \cdot v|^2 e^{- \frac{\pi}{\alpha} |u|^2} + \sum_{x \in \AD \setminus \{0\}} |(x+u) \cdot v|^2 e^{- \frac{\pi}{\alpha} |x+u|^2}, \quad \Psiav(0) = \sum_{x \in \AD \setminus \{0\}} |x \cdot v|^2 e^{- \frac{\pi}{\alpha} |x|^2}.
\end{equation}
Until the end of the proof, we assume that $\alpha$ is “small” in the sense:
\begin{equation}
  \alpha \leq \alphac := 0.552
\end{equation}
and our goal is to prove, for some constant $\cc > 0$ independent of $\alpha, u$:
\begin{equation}
\label{tohold}
\Psiav(u) - \Psiav(0) \geq \cc e^{- \frac{\pi}{\alpha} \xstar^2}  |u|^2.
\end{equation}

The proof uses only elementary tools but is a bit intricate. We treat the first shell (and the origin) and the higher shells separately, and then distinguish several regimes in each case. 

\subsection{The first shell and the origin} 
Let $\Ss$ be the first shell of $\AD$. We denote by $\Psisas$ the contribution coming from $\Ss \cup \{0\}$:
\begin{equation*}
\Psisas : u \mapsto |u \cdot v|^2 e^{- \frac{\pi}{\alpha} |u|^2} + \sum_{s \in \Ss} |(s+u) \cdot v|^2 e^{- \frac{\pi}{\alpha} |s+u|^2}. 
\end{equation*}
Using the $2$-design property \eqref{eq:2design}, for every shell $S$ of radius $r$ we have
\begin{equation*}
\sum_{s \in S} |s \cdot v|^2 e^{- \frac{\pi}{\alpha} |s|^2} =  3 r^2 e^{- \frac{\pi}{\alpha} r^2},
\end{equation*}
so we are interested in finding a lower bound of the form:
\begin{equation}
\label{Psisasu0}
\Psisas(u) - \Psisas(0) = |u \cdot v|^2 e^{- \frac{\pi}{\alpha} |u|^2} + \sum_{s \in \Ss} |(s+u) \cdot v|^2 e^{- \frac{\pi}{\alpha} |s+u|^2} - 3 \xstar^2 e^{- \frac{\pi}{\alpha} \xstar^2} \geq \cc |u|^2 e^{- \frac{\pi}{\alpha} \xstar^2}.
\end{equation}
We will consider three cases: 
\begin{enumerate}
		\item If $|u \cdot v|$ is “large”, a direct computation suffices.
		\item If $|u \cdot v|$ is “small” and $|u|$ is “small”, we use a Taylor expansion.
		\item If $|u \cdot v|$ is “small” and  $|u|$ is “large”, direct or perturbative computations fail. We identify a “good vertex” within the first shell and show that its contribution is always big enough for \eqref{tohold} to hold. This relies on a tedious identification of the “worst case scenario” among all possible parameters. 
\end{enumerate}
\textit{Some of the arguments below rely on purely numerical inequalities between certain explicit quantities. A Python notebook that checks those inequalities \href{https://plmlab.math.cnrs.fr/tleble/hexagonal}{is available here}.}

\subsection{First situation: \texorpdfstring{$|u \cdot v|$}{<u,v>} is large.}
This first part deals only with the contribution of the origin. We make a simple observation:
\begin{claim}[If $|u \cdot v|$ is “large”.]
\label{claim:uvlarge}
There exists a constant $\cc > 0$ such that the following holds. Assume that $u, \alpha, v$ are such that:
\begin{equation}
\label{eq:uvlarge}
|u \cdot v| \geq \sqrt{3.01} \xstar e^{\frac{\pi}{2\alpha}\left(|u|^2 - \xstar^2\right) },
\end{equation}
then we have:
\begin{equation}
\label{LBQuadraticuvlarge}
\Psisas(u) - \Psisas(0) \geq \cc |u|^2 e^{-\frac{\pi}{\alpha} \xstar^2}.  
\end{equation}
\end{claim}
\begin{proof} Discarding the contribution from the first shell and keeping only the one from $0$ in \eqref{Psisasu0}, we get:
\begin{equation*}
\Psisas(u) - \Psisas(0) \geq |u \cdot v|^2 e^{- \frac{\pi}{\alpha} |u|^2}  - 3 \xstar^2 e^{- \frac{\pi}{\alpha} \xstar^2} = \left(|u \cdot v|^2 e^{\frac{\pi}{\alpha} \left(\xstar^2 - |u|^2 \right)} - 3 \xstar^2 \right) e^{- \frac{\pi}{\alpha} \xstar^2}. 
\end{equation*}
 Our assumption \eqref{eq:uvlarge} ensures that
\begin{equation*}
|u \cdot v|^2 e^{\frac{\pi}{\alpha} \left(\xstar^2 - |u|^2 \right)} \geq 3.01 \xstar^2, 
\end{equation*}
so in particular, for some small constant $\cc > 0$ we can guarantee that, as stated in \eqref{LBQuadraticuvlarge}:
\begin{equation*}
\Psisas(u) - \Psisas(0) \geq \cc \xstar^2  e^{- \frac{\pi}{\alpha} \xstar^2} \geq \cc |u|^2  e^{- \frac{\pi}{\alpha} \xstar^2}.
\end{equation*} 
\end{proof}
In the sequel, when considering $\Psisas$, we may thus assume that $|u \cdot v|$ is not “large” as in \eqref{eq:uvlarge}. 

\subsection{Second situation: \texorpdfstring{$|u \cdot v|$}{<u,v>} is small and \texorpdfstring{$|u|$}{u} is small.}
We now turn to the case where $|u|$ is small \emph{compared to $\alpha$}, in the sense that $\frac{\pi}{\alpha} |u| \leq \mustar$, for a certain threshold $\mustar$.
\begin{lemma}
\label{lem:pialphauverysmall} 
There exists $\cc > 0$ such that the following holds.
Assume that $u,v, \alpha$ are such that:
\begin{equation}
\label{eq:uvsmallusmall}
\alpha \leq \alphac, \quad |u \cdot v| \leq \sqrt{3.01} \xstar e^{\frac{\pi}{2 \alpha}\left(|u|^2 - \xstar^2\right) }, \quad \frac{\pi}{\alpha} |u| \leq \mustar = 2.73.
\end{equation}
Then we have:
\begin{equation}
\label{Psisasus0usmall}
\Psisas(u) - \Psisas(0) \geq \cc |u|^2 e^{- \frac{\pi}{\alpha} \xstar^2}.
\end{equation}
\end{lemma}
\begin{proof}
We will prove that
\begin{equation}
\label{Ssusmall}
\sum_{s \in \Ss} |(s+u) \cdot v|^2 e^{- \frac{\pi}{\alpha} |s+u|^2} -  3 \xstar^2 \geq 0.01 |u|^2  e^{- \frac{\pi}{\alpha} \xstar^2}, 
\end{equation}
which clearly implies \eqref{Psisasus0usmall}. Expand the squares in the summands of \eqref{Ssusmall} and write, for $s \in \Ss$:
\begin{equation}
\label{expandsquareSDL}
((s+u) \cdot v)^2 e^{- \frac{\pi}{\alpha} |s+u|^2} = \left( (s \cdot v)^2  + 2 (s \cdot v) (u \cdot v) + (u \cdot v)^2 \right) e^{- \frac{2\pi}{\alpha} s \cdot u} e^{- \frac{\pi}{\alpha} |u|^2 - \frac{\pi}{\alpha} \xstar^2}.
\end{equation}
We focus first on the $s$-dependent part of the sum, namely:
\begin{equation}
\label{ExpandSUM}
\SUM := \sum_{s \in \Ss} \left( (s \cdot v)^2  + 2 (s \cdot v) (u \cdot v) + (u \cdot v)^2 \right) e^{- \frac{2\pi}{\alpha} s \cdot u}.
\end{equation} 

\subsubsection*{Step 1. A Taylor expansion to third order.}
Using the following elementary lower bound:
\begin{equation*}
e^{- \frac{2\pi}{\alpha} s \cdot u} \geq \sum_{k=0}^3 \frac{1}{k!} \left(- \frac{2\pi}{\alpha} s \cdot u\right)^k,
\end{equation*}
we see that $\SUM \geq \FDLT$, where $\FDLT$ is the “third order Taylor Expansion” of $\SUM$, namely
\begin{equation}
\label{FDLT}
\FDLT := \sum_{k=0}^3 \frac{1}{k!}  \sum_{s \in \Ss} \left( (s \cdot v)^2  + 2 (s \cdot v) (u \cdot v) + |u \cdot v|^2 \right)\left(- \frac{2\pi}{\alpha} s \cdot u\right)^k.
\end{equation}
By symmetry of $\Ss$ with respect to $s \mapsto -s$, the sums with an odd total power in $s$ vanish. Moreover, we can use the following identities due to the symmetries of a regular hexagon:
\begin{itemize}
  \item By the $2$-design property \eqref{eq:2design}, we have: 
  \begin{equation*}
\sum_{s \in \Ss}  (s \cdot v)^2 = 3 \xstar^2, \quad \sum_{s \in \Ss} (s \cdot u)^2 = 3 \xstar^2 |u|^2, \quad  \sum_{s \in \Ss} (s \cdot v) (s \cdot u) = 3 \xstar^2 (u \cdot v)
  \end{equation*}
  \item By the $4$-design property stated in Remark \ref{rem:5design}, we have
  \begin{multline*}
\sum_{s \in \Ss} (s \cdot v)^2 (s \cdot u)^2 = \frac{3}{4} \xstar^4 |u|^2 + \frac{3}{2} \xstar^4 (u \cdot v)^2, \quad , \quad \sum_{s \in \Ss} (s \cdot v) (s \cdot u)^3 = \frac{9}{4} \xstar^4 (u \cdot v) |u|^2, \quad \sum_{s \in \Ss} (s \cdot u)^4 = \frac{9}{4} \xstar^4 |u|^4. 
\end{multline*}
\end{itemize}
In the sequel, we use
\begin{equation*}
\K := \frac{\pi}{\alpha}
\end{equation*}
Expanding the right-hand side of \eqref{FDLT}, and using the identities above, we get:
\begin{equation}
\label{FDLT2}
\FDLT = 3\xstar^2 + \frac{3}{2} \K^2 \xstar^4 |u|^2 + (u \cdot v)^2 \left(6 + 3 \K^2 \xstar^4 \left(1 - 2 \K |u|^2 \right) + 6 \K \xstar^2 \left( \K |u|^2 - 2 \right) \right).
\end{equation}
The following term might give a negative contribution:
\begin{equation*}
\Gamma := 6 + 3 \K^2 \xstar^4 \left(1 - 2 \K |u|^2 \right) + 6 \K \xstar^2 \left( \K |u|^2 - 2 \right).
\end{equation*}

\paragraph{Step 2. Preliminaries}
Using the first and third constraints in \eqref{eq:uvsmallusmall} and our notation, we have:
\begin{equation*}
\K \geq \K_0 := \frac{\pi}{\alphac} \approx 5.691, \quad |u| \leq \frac{\mustar}{\K}.
\end{equation*}
Moreover, from the second condition in \eqref{eq:uvsmallusmall} we know that:
\begin{equation*}
(u \cdot v)^2 \leq 3.01 \xstar^2 e^{\K \left(|u|^2 - \xstar^2\right)}
\end{equation*}
Using the variable $y := |u|^2$, we thus have:
\begin{equation*}
\FDLT \geq 3\xstar^2 + \frac{3}{2} \K^2 \xstar^4 y + A(\K, y) \min\left(0, \Gamma(\K, y) \right),
\end{equation*}
where 
\begin{equation*}
y \leq \ymax(\K) := \frac{\mustar^2}{\K^2}, \quad A(\K, y) := 3.01 \xstar^2 e^{\K \left(y - \xstar^2\right)}, \quad \Gamma(\K, y) := \left(6 +  3 \K^2 \xstar^4 \left(1 - 2 \K y \right) + 6 \K \xstar^2 \left( \K y - 2 \right)  \right).
\end{equation*}
Observe that:
\begin{itemize}
  \item $\partial_y \Gamma$ is always negative, because $\K \xstar^2 \geq \K_0 \xstar^2 > 1$, and $\partial^2_{yy} \Gamma = 0$.
  \item $\Gamma(\K, 0) = 3\left(2 + \K^2 \xstar^4 - 4 \K \xstar^2 \right)$ is positive because $\K \geq \K_0 \geq \frac{\left(2 + \sqrt{2}\right)}{\xstar^2}$, which is the largest root.
  \item $\partial_y A = \K A$ and $\partial^2_{yy} A = \K^2 A$
\end{itemize}

\paragraph{Step 3. Concavity}
Consider the function
\begin{equation*}
\Phi(\K, y) := 3\xstar^2 + \frac{3}{2} \K^2 \xstar^4 y + A(\K, y) \min\left(0, \Gamma(\K, y) \right) - 3 \xstar^2 e^{\K y} - 0.01 y e^{\K y}.
\end{equation*}
We want to prove that $\Phi \geq 0$ for all admissible values of $\K, y$. Clearly, we have $\Phi(\K, 0) = 0$ for all $\K$. We claim that $\Phi$ is concave with respect to the second variable. Indeed:
\begin{itemize}
  \item On $\{\Gamma \geq 0\}$, we have:
\begin{equation*}
\partial^2_{yy} \Phi = - 3\xstar^2 \K^2 e^{\K y} - 0.01 \left(2 \K e^{\K y} + \K^2 y e^{\K y} \right) \leq 0.
\end{equation*}
  \item On $\{\Gamma \leq 0\}$, we have:
\begin{equation*}
\partial^2_{yy} \Phi = \partial^2_{yy} A \Gamma + 2 \partial_y A \partial_y \Gamma  - 0.01 \left(2 \K e^{\K y} + \K^2 e^{\K y} \right) = \K A \left(\K  \Gamma + 2 \partial_y \Gamma  \right) - 0.01 \left(2 \K e^{\K y} + \K^2 y e^{\K y} \right).
\end{equation*}
Here $\Gamma \leq 0$, and we know that $\partial_y \Gamma \leq 0$, so clearly $\partial^2_{yy} \Phi \leq 0$.
\end{itemize}
Since $\Phi(\K, \cdot)$ is concave on $[0, \ymax(\K)]$, it is enough to prove that $\Phi(\K, \ymax(\K)) \geq 0$.

\paragraph{Step 4. Study at $y = \ymax$.}
Plugging $y = \ymax = \frac{\mustar^2}{\K^2}$ into the expression of $\Phi$, we are left to show that:
\begin{equation*}
3\xstar^2 + \frac{3}{2} \xstar^4 \mustar^2 + 3.01 \xstar^2 e^{\frac{\mustar^2}{\K} - \K \xstar^2}  \min\left(0, \Gamma(\K, \ymax) \right) - 3 \xstar^2 e^{\frac{\mustar^2}{\K}} - 0.01 \frac{\mustar^2}{\K^2} e^{\frac{\mustar^2}{\K}} \geq 0.
\end{equation*}
with $\Gamma(\K, \ymax)$ given by:
\begin{equation*}
\Gamma(\K, \ymax) =  6 +  3 \K^2 \xstar^4  - 6 \K \mustar^2 \xstar^4 + 6 \xstar^2 \mustar^2 - 12 \K \xstar^2.
\end{equation*}
Observe that:
\begin{itemize}
  \item The quantity $e^{\frac{\mustar^2}{\K} - \K \xstar^2}$ is decreasing on $[\K_0, + \infty)$
  \item The quantity $3 \xstar^2 e^{\frac{\mustar^2}{\K}} + 0.01 \frac{\mustar^2}{\K^2} e^{\frac{\mustar^2}{\K}}$ is decreasing on $[\K_0, + \infty)$
  \item The map $\K \mapsto 6 +  3 \K^2 \xstar^4  - 6 \K \mustar^2 \xstar^4 + 6 \xstar^2 \mustar^2 - 12 \K \xstar^2$ is bounded below by $-3\left(2  + 2 \xstar^2 \mustar^2 + \xstar^4 \mustar^4 \right)$.
\end{itemize}
It is thus enough to check that:
\begin{equation*}
3\xstar^2 + \frac{3}{2} \xstar^4 \mustar^2 - 3.01 \xstar^2 e^{\frac{\mustar^2}{\K_0} - \K_0 \xstar^2} \times 3 \left(2  + 2 \xstar^2 \mustar^2 + \xstar^4 \mustar^4 \right) - 3 \xstar^2 e^{\frac{\mustar^2}{\K_0}} - 0.01 \frac{\mustar^2}{\K_0^2} e^{\frac{\mustar^2}{\K_0}} \geq 0.
\end{equation*}
This is a purely numerical inequality which is true with our choices $\alphac = 0.552$ (and $\K_0 = \frac{\pi}{\alphac}$), $\mustar = 2.73$.
\end{proof}

\renewcommand{\lambdac}{\mustar}

\subsection{Third situation: \texorpdfstring{$|u|$}{u} is “large”, and \texorpdfstring{$|u \cdot v|$}{<u, v>} is “small”.}
This last regime is the most difficult one to study. The following lemma, whose proof occupies the next few pages, deals with the first shell. 
\begin{lemma}
\label{lem:uvsmallularge}
There exists a constant $\cc > 0$ such that the following holds. Assume that $u, \alpha, v$ are such that:
\begin{equation}
\label{eq:uvsmallularge}
|u \cdot v| \leq \sqrt{3.01} \xstar e^{\frac{\pi}{2 \alpha}\left(|u|^2 - \xstar^2\right) }, \quad \frac{\pi}{\alpha} |u| \geq \mustar = 2.73.
\end{equation}
Then we have:
\begin{equation*}
\Psisas(u) - \Psisas(0) \geq \cc |u|^2 e^{- \frac{\pi}{\alpha} \xstar^2}.
\end{equation*}
\end{lemma}
\begin{proof}
From \eqref{Psisasu0}, we know that:
\begin{equation}
\label{Psiasu0debutcomplique}
\Psisas(u) - \Psisas(0) \geq \sum_{s \in \Ss} |(s+u) \cdot v|^2 e^{- \frac{\pi}{\alpha} |s+u|^2} - 3 \xstar^2 e^{- \frac{\pi}{\alpha} \xstar^2}.
\end{equation}
Studying the full sum over the six vertices is challenging, but fortunately we can show that the contribution of \emph{a single “good” vertex} of $\Ss$ is sufficient to beat the negative term. 

\paragraph{Choice of the good vertex.} Denote the vertex of $\Ss$ in the direction of $(1,0)$ by $s_1$ and label the other vertices as shown on the picture below. Without loss of generality, we can assume that the angle $\theta$ between $s_1$ and the vector $u$ is in $[0, \frac{\pi}{6}]$, in which case we select $s_5$ as our “good vertex”.
\begin{figure}[h!]
\begin{center}
\begin{tikzpicture}[scale=1.7, line cap=round, line join=round]
  \def\R{1.3457}     
  \def\uRel{0.62}     
  \def\thetaDeg{18}   
  \def\vDelta{103}     
  \def\markrad{0.045} 

  \pgfmathsetmacro{\uLen}{\uRel*\R}

  \coordinate (O) at (0,0);

  \draw[thin] (O) circle (1);

  \foreach \ang/\name in {0/s1,60/s2,120/s3,180/s4,240/s5,300/s6} {
    \coordinate (\name) at (\ang:\R);
  }

  \foreach \P/\lab in {s1/$s_1$, s2/$s_2$, s3/$s_3$, s4/$s_4$, s5/$s_5$, s6/$s_6$} {
    \fill (\P) circle (0.015);
    \node[font=\footnotesize, anchor=west] at ($(\P)+(0.02,0)$) {\lab};
  }

  \draw[red, line width=0.7pt] (s5) circle (\markrad);

  \coordinate (U) at (\thetaDeg:\uLen);
  \draw[-{Latex[length=2.0mm]}, thick] (O) -- (U) node[below right = -0.3pt] {$u$};

  \pgfmathsetmacro{\phiDeg}{\thetaDeg + \vDelta}
  \coordinate (V) at (\phiDeg:1);
  \draw[-{Latex[length=2.0mm]}, thick] (O) -- (V) node[pos=0.55, above left=1pt] {$v$};

  \coordinate (X) at (\R,0); 
  \pic[draw,->,angle radius=9mm,angle eccentricity=1.15,
       "$\theta$"{font=\scriptsize, inner sep=1pt}] {angle = X--O--U};
\end{tikzpicture}
\end{center}
\caption{Without loss of generality, we assume that $\theta \in \left[0, \frac{\pi}{6}\right]$.}
\label{figures4}
\end{figure}
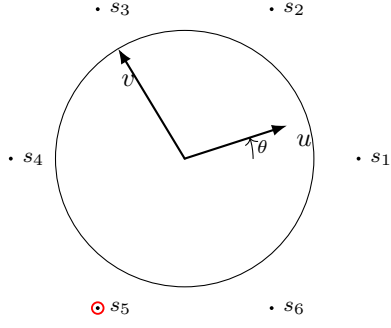

Focus on the contribution of $s_5$, expand the squares in the exponent, and write
\begin{equation}
\label{psiasu0uvsmallB}
|(s_5+u) \cdot v|^2 e^{- \frac{\pi}{\alpha} |s_5+u|^2} = |(s_5+u) \cdot v|^2 e^{- \frac{2\pi}{\alpha} s_5 \cdot u - \frac{\pi}{\alpha} |u|^2} e^{- \frac{\pi}{\alpha} \xstar^2}.
\end{equation}
We want this to beat $3 \xstar^2 e^{- \frac{\pi}{\alpha} \xstar^2}$. In that regard, the vertex $s_5$ has two interesting features:
\begin{itemize}
  \item It is “opposite” to $u$, so the term $- \frac{2\pi}{\alpha} s_5 \cdot u$ in the exponent should be “very positive” ($\alpha$ being “small”).
  \item It is \emph{not aligned with u}, so that $(s_5+u) \cdot v$ should not be too small. 

  Indeed, the first condition in \eqref{eq:uvsmallularge} tells us that $|u \cdot v|$ is “small”, thus $u$ and $v$ are “almost orthogonal”. If we chose the vertex $s_4$, which is “even more opposite” to $u$, there would be a risk to have $(s_4+u) \cdot v \approx 0$.
\end{itemize}
The rest of the proof of Lemma \ref{lem:uvsmallularge} consists in showing that, with our conditions \eqref{eq:uvsmallularge}, we always have
\begin{equation}
\label{s5isgood}
|(s_5+u) \cdot v|^2 e^{- \frac{2\pi}{\alpha} s_5 \cdot u - \frac{\pi}{\alpha} |u|^2} \geq 3 \xstar^2 + \cc |u|^2.
\end{equation} 

\paragraph{Step 0. Admissible parameters.}
\begin{itemize}
  \item Write $u$ as $u = (\rho \cos \theta, \rho \sin \theta)$. We have the following constraints:
\begin{equation*}
\rho \leq \frac{\xstar}{\sqrt{3}} \text{ (because $u \in \HH$)}, \quad \theta \in \left[0, \frac{\pi}{6}\right].
\end{equation*}
\item We recall that $\alpha$ is such that 
\begin{equation*}
\alpha \leq \min\left(\alphac, \frac{\pi \rho}{\lambdac} \right).
\end{equation*}

\item The vector $v$ has norm $1$ and must satisfy the condition
\begin{equation}
\label{uvbalpha}
|u \cdot v| \leq b(\alpha) :=  \sqrt{3.01} \xstar e^{\frac{\pi}{2 \alpha} \left(|u|^2 - \xstar^2\right)}.
\end{equation}
This condition is “increasing” in $\alpha$ in the sense that for a fixed $u$, there are more admissible $v$'s when $\alpha$ increases.
\end{itemize}
We consider the following function of $u, \alpha, v$:
\begin{equation}
\label{def:Fm}
\Fm(u, \alpha, v) := |(s_5+u) \cdot v|^2 e^{- \frac{2\pi}{\alpha} s_5 \cdot u - \frac{\pi}{\alpha} |u|^2}.
\end{equation}
Our goal is to prove that, for all admissible values of $u, \alpha, v$, we have:
\begin{equation*}
\Fm(u, \alpha, v) \geq 3 \xstar^2 + \cc |u|^2.
\end{equation*}

\paragraph{Step 1. Worst case for $v$.}
We start by minimizing the prefactor $|(s_5+u) \cdot v|^2$ over $v$ (there is no $v$-dependency in the exponential term).
\begin{claim}
\label{claim:muak}
Let $u \in \R^2$ be a nonzero vector, let $a \in \R^2$, and $\kappa \in [0,1]$. Let $z_1 := \frac{u}{|u|}$ and choose $z_2$  such $(z_1, z_2)$ is an orthonormal basis of $\R^2$. Let $a_1, a_2$ be the coordinates of $a$ in this basis. We have:
\begin{equation}
\label{def:muak}
\min_{|v| = 1, |u \cdot v| \leq \kappa |u|} |a \cdot v| = \max\left(0, |a_2| \sqrt{1-\kappa^2} - |a_1| \kappa \right).
\end{equation}
\end{claim}
\begin{proof}
 Write $v$ in the basis $(z_1, z_2)$ as $v=(\cos\varphi, \sin\varphi)$ for some $\varphi$. We must then minimize the quantity $\lvert a\cdot v\rvert=\lvert a_1\cos\varphi+a_2\sin\varphi\rvert$ under the constraint $\lvert u\cdot v\rvert\le \kappa |u|$, which is equivalent to $\lvert\cos\varphi\rvert\le \kappa$.

For a fixed value of $t := \lvert\cos\varphi\rvert \in [0, \kappa]$, the quantity $\lvert a_1\cos\varphi+a_2\sin\varphi\rvert$ is minimal when the two summands have opposite signs (so they subtract), in which case it is equal to $\left| |a_2| \sqrt{1-t^2} - |a_1| t \right|$. Therefore:
\begin{equation*}
\min_{|v| = 1, |u \cdot v| \leq \kappa |u|} |a \cdot v| = \min_{t\in[0,\kappa]} h(t), \qquad h(t):= \left| |a_2| \sqrt{1-t^2} - |a_1| t \right|.
\end{equation*}
The function $g := t \mapsto |a_2| \sqrt{1-t^2} - |a_1| t$ is strictly decreasing on $[0,1]$, with $g(0)=\lvert a_2\rvert\ge 0$, and $g(\kappa)=\lvert a_2\rvert\sqrt{1-\kappa^2}-\lvert a_1\rvert\,\kappa$. There are two cases:
\begin{itemize}
   \item If $g(\kappa)\ge 0$. Then $g(t)\ge 0$ for all $t\in[0,\kappa]$, and $h(t)=g(t) \geq g(\kappa)$ on $[0,\kappa]$.
\item If $g(\kappa)\le 0$. Since $g(0)\ge 0$ and $g$ is continuous and strictly decreasing, there exists
$t_*\in[0,\kappa]$ with $g(t_*)=0$. Then $h(t_*)=\lvert g(t_*)\rvert=0$, and thus $\min_{t\in[0,\kappa]}h(t)=0$.
 \end{itemize} 
Combining both cases yields the result.
\end{proof}

In particular, choosing $a = s_5 + u$, so that:
\begin{multline*}
a_1 = (s_5 + u) \cdot \frac{u}{|u|} = \xstar \cos\left(\pi + \frac{\pi}{3} - \theta\right) + \rho = \rho - \xstar \cos\left(\frac{\pi}{3} - \theta\right), \\
 |a_2| = \xstar \left|\sin\left(\pi + \frac{\pi}{3} - \theta\right)\right| = \xstar \sin\left(\frac{\pi}{3} - \theta\right)
\end{multline*} 
and letting $\kappa = \kappa(\alpha, \rho)$, where 
\begin{equation*}
\kappa(\alpha, \rho) := \frac{b(\alpha)}{\rho}, \text{ with } b(\alpha) = \sqrt{3.01} \xstar e^{\frac{\pi}{2 \alpha} \left(|u|^2 - \xstar^2\right)} \text{ as in \eqref{uvbalpha}},
\end{equation*}
applying \eqref{def:muak} yields:
\begin{multline}
\label{eq:MMualpha}
\MM(u, \alpha) := \min_{|v| = 1, |u \cdot v| \leq b(\alpha)} |(s_5+u) \cdot v| \\ = \max\left(0,  \xstar \sin\left(\frac{\pi}{3}- \theta \right)\sqrt{1 - \left(\kappa(\alpha, \rho)\right)^2} - \left|\rho - \xstar \cos\left(\frac{\pi}{3} - \theta\right) \right| \kappa(\alpha, \rho) \right),
\end{multline}
and in particular we can write for all admissible $u, \alpha, v$:
\begin{equation}
\label{def:GmFm}
\Fm(u, \alpha, v) \geq \Gm(u, \alpha) := \MM(u, \alpha)^2 e^{- \frac{2\pi}{\alpha} s_5 \cdot u - \frac{\pi}{\alpha} |u|^2}.
\end{equation}

\paragraph{Step 2. Worst case for $\alpha$.}
\begin{claim}
\label{claim:alphamonotone}
For all admissible $u, \alpha$, we have:
\begin{equation}
\label{alphamax}
\Gm(u, \alpha) \geq \Gm\left(u, \alphamax(\rho)\right), \quad \alphamax(\rho):=\min\left(\alphac,\ \frac{\pi \rho}{\lambdac}\right).
\end{equation}
\end{claim}
\begin{proof} It follows from the fact that the map $\alpha \mapsto \Gm(u, \alpha)$ is \emph{decreasing}. Indeed:
\begin{itemize}
  \item Write $s_5 \cdot u$ as $-\xstar \rho \cos\left(\frac{\pi}{3} - \theta\right)$. Since $\theta\in[0,\pi/6]$, we have $\cos\left(\frac{\pi}{3} - \theta\right) \ge \tfrac12$, hence
\begin{equation}
\label{eq:expweight}
 2\xstar \rho \cos\left(\frac{\pi}{3} - \theta\right)  - \rho^2 \geq \xstar \rho - \rho^2 \geq 0,
\end{equation}
because $\rho \leq \frac{1}{\sqrt{3}} \xstar$, thus the exponential term is decreasing in $\alpha$.
\item The quantity $b(\alpha)$ defined in \eqref{uvbalpha} is increasing in $\alpha$, hence the constraint $|u \cdot v| \leq b(\alpha)$ \emph{weakens} as $\alpha$ increases, and thus, for fixed $u$, the quantity $\MM(u, \alpha)$, defined as a minimum, is \emph{decreasing} in $\alpha$.
\end{itemize}
The maximal admissible $\alpha$ must satisfy both $\alpha \leq \alphac$ and $\alpha \leq \frac{\pi \rho}{\lambdac}$, hence the expression \eqref{alphamax} for $\alphamax$.
\end{proof}

For $u = (\rho\cos \theta, \rho\sin \theta)$, the quantity $\Gm\left(u, \alphamax(\rho)\right)$ is given, with $\MM(u, \alphamax(\rho))$ as in \eqref{eq:MMualpha}, by:
\begin{equation}
\label{Gmalphamax}
\Gm\left(u, \alphamax(\rho)\right) = \MM(u, \alphamax(\rho))^2 e^{\frac{2\pi}{\alphamax(\rho)}\left(2\xstar \rho \cos\left(\frac{\pi}{3} - \theta\right)  -  \rho^2\right)}.
\end{equation}
In the following, we will denote by $\kappa(\rho)$ the quantity appearing in $\MM(u, \alphamax(\rho))$:
\begin{equation}
\label{def:kappar}
\kappa(\rho) := \kappa(\alphamax(\rho), \rho) := \frac{b(\alphamax(\rho))}{\rho} :=  \sqrt{3.01}  \frac{\xstar}{\rho} e^{\frac{\pi}{2 \alphamax(\rho)}\left(\rho^2 - \xstar^2\right)}, \text{ with $\alphamax(\rho):=\min\left(\alphac,\ \frac{\pi \rho}{\lambdac}\right)$.}
\end{equation}

\paragraph{Step 3. A quick study of $\kappa$.}
With our choice for $\lambdac$, the following holds.
\begin{claim}
\label{claim:studykappa} The map $\rho \mapsto \kappa(\rho)$ is increasing over $\left[0, \frac{\xstar}{\sqrt{3}}\right]$, and we have:
\begin{equation*}
\max_{\rho \in \left[0, \frac{\xstar}{\sqrt{3}}\right]} \kappa\left(\rho\right) = \kappa\left(\frac{\xstar}{\sqrt{3}}\right) = \sqrt{3.01} \times \sqrt{3} e^{- \frac{\pi}{3 \alphac} \xstar^2} \leq 0.337 < \hal.
\end{equation*} 
\end{claim}
\begin{proof}
In view of our expression for $\alphamax(\rho)$, define the threshold $\rho_0$ as:
\begin{equation}
\label{def:rho0}
\rho_0 := \frac{\lambdac \alphac}{\pi},
\end{equation} 
so that if $\rho \leq \rho_0$, then $\alphamax(\rho) = \frac{\pi \rho}{\lambdac}$ and if $\rho \geq \rho_0$, then $\alphamax(\rho) = \alphac$.
\begin{itemize}
  \item If $\rho \leq \rho_0$, the expression for $\kappa(\rho)$ (see \eqref{def:kappar}) is $\sqrt{3.01} \frac{\xstar}{\rho} e^{\frac{\lambdac}{2}\left(\rho - \frac{\xstar^2}{\rho}\right)}$. Positivity of its derivative is equivalent to $\frac{\lambdac}{2}\rho^2 - \rho + \frac{\mustar}{2} \xstar^2 \geq 0$, which is true because the discriminant is negative, indeed:
  \begin{equation*}
\mustar \geq 1 \geq  \frac{1}{\xstar}.
  \end{equation*}
  \item If $\rho \geq \rho_0$, the expression for $\kappa(\rho)$ is $\sqrt{3.01} \frac{\xstar}{\rho} e^{\frac{\pi}{2 \alphac}\left(\rho^2 - \xstar^2\right)}$. This is increasing if $\rho \geq \sqrt{\frac{\alphac}{\pi}}$, but here $\rho \geq \rho_0 = \frac{\lambdac \alphac}{\pi}$ and we do have $\mustar \geq \sqrt{\frac{\pi}{\alphac}}$.
\end{itemize}
This shows that $\rho \mapsto \kappa(\rho)$ is increasing. Finally, we simply compute
\begin{equation*}
\kappa\left( \frac{\xstar}{\sqrt{3}} \right) = \sqrt{3.01} \times \sqrt{3} e^{- \frac{\pi}{3 \alphac} \xstar^2} \leq 0.337 < \hal.
\end{equation*}
\end{proof}
It will be relevant in the sequel to notice that, with our choice for $\mustar$, we have:
\begin{equation*}
\rho_0 = \frac{\mustar \alphac}{\pi} \leq \frac{\xstar}{2}.
\end{equation*}

\paragraph{Step 4. Checking the positivity of the prefactor.}
\begin{claim}
\label{claim:positivityMM}
The quantity $\xstar \sin\left(\frac{\pi}{3} - \theta \right)\sqrt{1 - \left(\kappa(\rho)\right)^2} - \left|\rho - \xstar \cos\left(\frac{\pi}{3} - \theta\right) \right| \kappa(\rho)$ is always positive.
\end{claim}
\begin{proof}
Since $\theta\in[0,\pi/6]$, we have $\sin\left(\frac{\pi}{3} - \theta\right) \geq \hal$, as well as $0 \leq \cos\left(\frac{\pi}{3} - \theta\right) \leq \frac{\sqrt{3}}{2}$, which implies:
\begin{equation*}
\left|\rho - \xstar \cos\left(\frac{\pi}{3} - \theta\right) \right| \leq \max\left(\rho, \xstar \cos\left(\frac{\pi}{3} - \theta\right)\right) \leq \max\left(\frac{\xstar}{\sqrt{3}}, \frac{\sqrt{3}}{2} \xstar  \right) = \frac{\sqrt{3}}{2} \xstar.
\end{equation*}
It is thus enough to show that $\left(\sqrt{1-\kappa(\rho)^2}- \sqrt{3} \kappa(\rho) \right) \ge 0$ for all admissible $\rho$. This is equivalent to having $\kappa(\rho) \leq \hal$, which we know is true.
\end{proof}
As a consequence, the prefactor $\MM(u, \alphamax(\rho))$ in \eqref{Gmalphamax} is always given, cf. \eqref{eq:MMualpha}, by
\begin{equation}
\label{MMualphamaxr}
\MM(u, \alphamax(\rho)) = \xstar \sin\left(\frac{\pi}{3} - \theta \right)\sqrt{1 - \left(\kappa(\rho)\right)^2} - \left|\rho - \xstar \cos\left(\frac{\pi}{3} - \theta\right) \right| \kappa(\rho). 
\end{equation}

\paragraph{Step 5. Re-expressing $\Gm\left(u, \alphamax(\rho)\right)$}
Introduce the following additional variables:
\begin{equation*}
\phi := \frac{\pi}{3} - \theta \in \left[\frac{\pi}{6}, \frac{\pi}{3}\right], \quad \delta(\rho) := \arcsin\left( \kappa(\rho) \right)
\end{equation*}
In view of \eqref{MMualphamaxr} we can write, using the variable $\phi$:
\begin{equation}
\label{MMMrphigeneral}
\MMM(\rho, \phi) := \MM(u, \alphamax(\rho)) =  \xstar \sin\left(\phi \right)\sqrt{1 - \left(\kappa(\rho)\right)^2} - \left|\rho - \xstar \cos\left(\phi \right) \right| \kappa(\rho), 
\end{equation}
the only ambiguity being the sign of $\rho - \xstar \cos\left(\phi \right)$. There are several cases:
\begin{itemize}
  \item If $\rho \leq \frac{\xstar}{2}$, then for all $\phi \in \left[\frac{\pi}{6}, \frac{\pi}{3}\right]$ we have
  \begin{equation}
  \label{rleqxstar2}
\MMM(\rho, \phi) =  \xstar \sin\left(\phi \right)\sqrt{1 - \left(\kappa(\rho)\right)^2} - \xstar \cos\left(\phi \right) \kappa(\rho) + \rho\kappa(\rho) = \xstar \sin(\phi - \delta(\rho)) + \rho\kappa(\rho).
  \end{equation}
  \item If $\rho \geq \frac{\xstar}{2}$ and $\frac{\pi}{6} \leq \phi \leq \arccos\left( \frac{\rho}{\xstar}\right)$, we have again:
      \begin{equation*}
\MMM(\rho, \phi) =  \xstar \sin(\phi - \delta(\rho)) + \rho\kappa(\rho).
  \end{equation*}
  \item If $\rho \geq \frac{\xstar}{2}$ and  $\arccos\left( \frac{\rho}{\xstar}\right) \leq \phi \leq \frac{\pi}{3}$, we have instead:
\begin{equation*}
\MMM(\rho, \phi) =  \xstar \sin\left(\phi \right)\sqrt{1 - \left(\kappa(\rho)\right)^2} + \xstar \cos\left(\phi \right) \kappa(\rho) - \rho \kappa(\rho) = \xstar \sin(\phi + \delta(\rho)) - \rho\kappa(\rho).
\end{equation*}
\end{itemize}
We are left to study the following map over $[0, \frac{\xstar}{\sqrt{3}}] \times [\frac{\pi}{6}, \frac{\pi}{3}]$, cf. \eqref{Gmalphamax}
\begin{equation*}
\Hm : (\rho, \phi) \mapsto \MMM(\rho, \phi)^2 \exp\left(\frac{\pi}{\alphamax(\rho)} \left( 2\xstar \rho \cos \phi - \rho^2  \right) \right).
\end{equation*}

\paragraph{Step 5. Worst case for $\phi$.} 
\begin{claim}
For all admissible $\rho$, the quantity $\Hm(\rho, \phi)$ is minimal at $\phi = \frac{\pi}{3}$.
\end{claim}
\begin{proof}
The proof goes in two steps: first, we show that the minimum is always at $\frac{\pi}{3}$ \emph{or} $\frac{\pi}{6}$, then we compare the two values directly and find that the smaller one is always at $\frac{\pi}{3}$.

\vspace{0.2cm}
\textit{1. The minimum is at the endpoints.} It suffices to show that $\partial_\phi \Hm(\rho, \cdot)$ is decreasing for fixed $\rho$. Write:
\begin{equation*}
\partial_\phi  \Hm(\rho, \phi) = \left( 2 \MMM(\rho, \phi) \partial_\phi \MMM(\rho, \phi) -  \frac{2\pi \xstar \rho}{\alphamax(\rho)} \MMM(\rho, \phi)^2 \sin \phi \right) \exp\left(\frac{\pi}{\alphamax(\rho)} \left( 2\xstar \rho \cos \phi - \rho^2  \right) \right),
\end{equation*}
and recall that $\MMM(\rho, \phi)$ is always positive, thus the sign of $\partial_\phi  \Hm(\rho, \phi)$ is the same as the sign of
\begin{equation*}
\TT(\rho, \phi) := \partial_\phi \MMM(\rho, \phi) - \frac{\pi \xstar \rho}{\alphamax(\rho)} \MMM(\rho, \phi) \sin \phi .
\end{equation*}
We split the discussion into several cases. 
\begin{itemize}
  \item If $\rho \leq \frac{\xstar}{2}$, we know from \eqref{rleqxstar2} that
\begin{equation*}
\MMM(\rho, \phi) =  \xstar \sin(\phi - \delta) + \rho\kappa(\rho), \text{ and thus } \partial_\phi \MMM(\rho, \phi) = \xstar \cos(\phi - \delta),
\end{equation*}
in which case we obtain the following expression for $\TT(\rho, \phi)$:
\begin{equation}
\label{caseAoptimphi}
\TT(\rho, \phi) =  \xstar \cos(\phi - \delta) - \frac{\pi \xstar \rho}{\alphamax(\rho)} \left( \xstar \sin(\phi - \delta) + \rho\kappa(\rho)  \right) \sin \phi.
\end{equation}
Since $0 \leq \delta \leq \frac{\pi}{6}$, this is strictly decreasing for $\phi \in [\frac{\pi}{6}, \frac{\pi}{3}]$.

\item If $\rho \geq \frac{\xstar}{2}$, the expression for $\MMM(\rho, \phi)$ changes when $\phi$ reaches $\arccos\left( \frac{\rho}{\xstar}\right)$, so we distinguish two intervals:
\begin{itemize}
  \item On $\left[\frac{\pi}{6}, \arccos\left( \frac{\rho}{\xstar}\right)\right]$, the expression for $\TT(\rho, \phi)$ is as above, and still strictly decreasing in $\phi$. 
  \item On $\left[\arccos\left( \frac{\rho}{\xstar}\right), \frac{\pi}{3}\right]$, we have instead:
  \begin{equation*}
\MMM(\rho, \phi) =  \xstar \sin(\phi + \delta) - \rho\kappa(\rho), \text{ and thus } \partial_\phi \MMM(\rho, \phi) = \xstar \cos(\phi + \delta),
  \end{equation*}
 which yields the following expression for $\TT(\rho, \phi)$
  \begin{equation*}
\TT(\rho, \phi) =  \xstar \cos(\phi + \delta) -  \frac{\pi \xstar \rho}{\alphamax(\rho)} \left(  \xstar \sin(\phi + \delta) - \rho\kappa(\rho) \right) \sin \phi,
  \end{equation*}
  which is again strictly decreasing in $\phi$.
\end{itemize}
\end{itemize}
\vspace{0.2cm}
\textit{2. The minimum is at $\frac{\pi}{3}$.} We now compare the values of $\Hm(\rho, \frac{\pi}{6})$ and  $\Hm(\rho, \frac{\pi}{3})$ in different regimes, and show that:
\begin{equation}
\label{compareHmpi6pi3}
\frac{\Hm(\rho, \frac{\pi}{6})}{\Hm(\rho, \frac{\pi}{3})} = \left( \frac{\MMM(\rho, \frac{\pi}{6})}{\MMM(\rho, \frac{\pi}{3})}  \right)^2 \exp\left(\frac{\pi}{\alphamax(\rho)} \xstar \rho \left(\sqrt{3}-1 \right) \right) \geq 1.
\end{equation}
First, observe that since we assume $\frac{\pi}{\alpha} \rho \geq \mustar$, we have:
\begin{equation*}
\exp\left(\frac{\pi}{\alphamax(\rho)} \xstar \rho \left(\sqrt{3}-1 \right) \right) \geq \exp\left(\mustar \xstar \left(\sqrt{3}-1 \right) \right) \approx 8.56
\end{equation*}
and we are left to prove that the ratio $\frac{\MMM(\rho, \frac{\pi}{6})}{\MMM(\rho, \frac{\pi}{3})}$ cannot be too small, so that:
\begin{equation*}
\left( \frac{\MMM(\rho, \frac{\pi}{6})}{\MMM(\rho, \frac{\pi}{3})}  \right)^2 \exp\left(\mustar \xstar \left(\sqrt{3}-1 \right) \right) \geq 1.
\end{equation*}

\begin{itemize}
  \item If $0 \leq \rho \leq \frac{9 \xstar}{20}$. Then of course $\rho \leq \hal \xstar$. According to the expressions for $\MMM$ given in \eqref{MMMrphigeneral}, we have:
  \begin{equation*}
\MMM\left(\rho, \frac{\pi}{6}\right) =  \frac{1}{2} \xstar \sqrt{1 - \left(\kappa(\rho)\right)^2} - \left(\frac{\sqrt{3}}{2} \xstar - \rho \right) \kappa(\rho), \quad \MMM\left(\rho, \frac{\pi}{3}\right) = \frac{\sqrt{3}}{2} \xstar \sqrt{1 - \left(\kappa(\rho)\right)^2} - \left(\hal \xstar - \rho \right) \kappa(\rho).
  \end{equation*}
For $\bkappa = \kappa(\rho)$ fixed, a quick computation shows that the map 
\begin{equation*}
\trho \mapsto \frac{\frac{1}{2} \xstar \sqrt{1 - \bkappa^2} - \left(\frac{\sqrt{3}}{2} \xstar - \trho \right) \bkappa}{\frac{\sqrt{3}}{2} \xstar \sqrt{1 - \bkappa^2} - \left(\hal \xstar - \trho \right) \bkappa}  = \frac{\left(\frac{1}{2} \xstar \sqrt{1 - \bkappa^2} - \frac{\sqrt{3}}{2} \xstar \bkappa\right) + \trho \bkappa}{\left(\frac{\sqrt{3}}{2} \xstar \sqrt{1 - \bkappa^2} - \hal \xstar  \bkappa\right) + \trho \bkappa}
\end{equation*}
is increasing, and thus minimal at $0$. We obtain:
\begin{equation*}
\frac{\MMM(\rho, \frac{\pi}{6})}{\MMM(\rho, \frac{\pi}{3})} \geq \frac{\sqrt{1 - \bkappa^2} - \sqrt{3} \bkappa}{\sqrt{3} \sqrt{1 - \bkappa^2} -   \bkappa},
\end{equation*}
and this quantity is decreasing in $\bkappa$. Since $\rho \mapsto \kappa(\rho)$ is increasing, we have $\bkappa \leq \kappa\left( \frac{9 \xstar}{20} \right)$, and thus:
\begin{equation*}
\frac{\MMM(\rho, \frac{\pi}{6})}{\MMM(\rho, \frac{\pi}{3})} \geq  \frac{\sqrt{1 - \kappa\left( \frac{9 \xstar}{20} \right)^2} - \sqrt{3} \kappa\left( \frac{9 \xstar}{20} \right)}{\sqrt{3} \sqrt{1 -  \kappa\left( \frac{9 \xstar}{20} \right)^2} -  \kappa\left( \frac{9 \xstar}{20} \right)}.
\end{equation*}
To prove \eqref{compareHmpi6pi3}, it remains to check the purely numerical inequality:
\begin{equation*}
\left( \frac{\sqrt{1 - \kappa\left( \frac{9 \xstar}{20} \right)^2} - \sqrt{3} \kappa\left( \frac{9 \xstar}{20} \right)}{\sqrt{3} \sqrt{1 -  \kappa\left( \frac{9 \xstar}{20} \right)^2} -  \kappa\left( \frac{9 \xstar}{20} \right)} \right)^2 \exp\left(\mustar \xstar \left(\sqrt{3}-1 \right) \right) \geq 1.
\end{equation*}

\item If $\frac{9 \xstar}{20} \leq \rho \leq \frac{1}{\sqrt{3}} \xstar$. Now, we write:
\begin{equation*}
\MMM\left(\rho, \frac{\pi}{6}\right) \geq \frac{1}{2} \xstar \sqrt{1 - \left(\kappa(\rho)\right)^2} -  \left(  \frac{\sqrt{3}}{2} - \frac{9}{20} \right) \xstar \kappa(\rho), 
\end{equation*}
and observe that even if the expression of $\MMM\left(\rho, \frac{\pi}{3}\right)$ changes at $\rho = \hal \xstar$, we always have:
\begin{equation*}
\MMM\left(\rho, \frac{\pi}{3}\right)  \leq \frac{\sqrt{3}}{2} \xstar \sqrt{1 - \left(\kappa(\rho)\right)^2}.
\end{equation*}
Dividing each term by $\hal \xstar$ gives the lower bound:
\begin{equation}
\label{910}
\frac{\MMM(\rho, \frac{\pi}{6})}{\MMM(\rho, \frac{\pi}{3})} \geq \frac{ \sqrt{1 -\kappa^2} -  \left( \sqrt{3} - \frac{9}{10} \right) \kappa }{ \sqrt{3}  \sqrt{1 -\kappa^2}}.
\end{equation}
This quantity is decreasing in $\kappa$ and thus (since $\kappa$ is increasing with $\rho$)
\begin{equation*}
\frac{\MMM(\rho, \frac{\pi}{6})}{\MMM(\rho, \frac{\pi}{3})} \geq \frac{ \sqrt{1 -\kappa\left(\frac{1}{\sqrt{3}} \xstar\right)^2} -  \left( \sqrt{3} - \frac{9}{10} \right) \kappa\left(\frac{1}{\sqrt{3}} \xstar\right) }{ \sqrt{3}  \sqrt{1 -\kappa\left( \frac{1}{\sqrt{3}} \xstar \right)^2}}.
\end{equation*}
To prove \eqref{compareHmpi6pi3}, it remains to check the purely numerical inequality:
\begin{equation*}
\left(\frac{ \sqrt{1 -\kappa\left(\frac{1}{\sqrt{3}} \xstar\right)^2} -  \left( \sqrt{3} - \frac{9}{10} \right) \kappa\left(\frac{1}{\sqrt{3}} \xstar\right) }{ \sqrt{3}  \sqrt{1 -\kappa\left( \frac{1}{\sqrt{3}} \xstar \right)^2}} \right)^2 \exp\left(\mustar \xstar \left(\sqrt{3}-1 \right) \right) \geq 1.
\end{equation*}
\end{itemize}
\end{proof}
As a consequence, for all admissible values of $\rho$, letting 
\begin{equation*}
\MMMt(\rho) := \MMM\left(\rho, \frac{\pi}{3}\right) = \frac{\sqrt{3}}{2} \xstar \sqrt{1 - \left(\kappa(\rho)\right)^2} - \left|\rho - \hal \xstar \right| \kappa(\rho), 
\end{equation*}
we obtain the following lower bound on $\Hm(\rho, \phi)$:
\begin{equation*}
\Hm(\rho, \phi) \geq \Im(\rho) := \Hm\left(\rho, \frac{\pi}{3}\right) = \MMMt(\rho)^2 \exp\left(\frac{\pi}{\alphamax(\rho)} \left( \xstar \rho - \rho^2  \right) \right).
\end{equation*}

\paragraph{Step 6. Worst case for $\rho$.}
\begin{claim}
For all admissible $\rho$, we have $\Im(\rho) \geq \Im\left( \frac{\xstar}{\sqrt{3}} \right)$.
\end{claim}
\begin{proof}
We consider three possibilities for $\rho$: between $0$ and $\rho_0$, between $\rho_0$ and $\hal \xstar$, between $\hal \xstar$ and $\frac{\xstar}{\sqrt{3}}$.
\begin{itemize}
\item If $0 \leq \rho \leq \rho_0 \leq \frac{\xstar}{2}$, we have:
\begin{equation*}
\MMMt(\rho)  = \left( \frac{\sqrt{3}}{2} \xstar \sqrt{1 - \left(\kappa(\rho)\right)^2} - \left(\hal \xstar - \rho\right) \kappa(\rho)  \right), \quad \exp\left(\frac{\pi}{\alphamax(\rho)} \left( \xstar \rho - \rho^2  \right) \right) = \exp\left(\lambdac \left( \xstar - \rho  \right) \right),
\end{equation*}
and thus the sign of $\Im'(\rho)$ is the same as the sign of $2 \MMMt'(\rho) - \lambdac \MMMt(\rho)$. We claim that $\Im'(\rho) \leq 0$. Indeed:
\begin{itemize}
  \item We know by Claim \ref{claim:studykappa} that $\rho \mapsto \kappa(\rho)$ is increasing and thus:
\begin{equation*}
\rho \mapsto \sqrt{1 - \left(\kappa(\rho)\right)^2} \text{ is decreasing,} \quad - \left(\hal \xstar - \rho\right) \kappa'(\rho) \text{ is negative,}
\end{equation*}
 in particular we see that $\MMMt'(\rho) \leq \kappa(\rho)$.
\item We have $\MMMt \geq  \frac{\sqrt{3}}{2} \xstar \sqrt{1 - \left(\kappa(\rho)\right)^2} - \hal \xstar \kappa(\rho)$
\end{itemize}
It remains to check that for all $0 \leq \kappa \leq \kappa\left(\frac{\xstar}{\sqrt{3}}\right)$ we have:
  \begin{equation*}
2 \kappa - \lambdac \left( \frac{\sqrt{3}}{2} \xstar \sqrt{1 - \kappa^2} - \hal \xstar \kappa \right) \leq 0,
  \end{equation*}
which boils down to a quadratic inequality in $\kappa$ that can be checked by hand.

\item If $\rho_0 \leq \rho \leq \frac{\xstar}{2}$, the expression of $\MMMt(\rho)$ does not change but $\exp\left(\frac{\pi}{\alphamax(\rho)} \left( \xstar \rho - \rho^2  \right) \right) = \exp\left(\frac{\pi}{\alphac} \left( \xstar \rho - \rho^2  \right) \right)$. Since $\rho \leq \hal \xstar$, this exponential term is increasing. We claim that $\MMMt(\rho)$ is also increasing, which will prove that $\Im(\rho)$ is increasing on this interval. We write:
\begin{equation*}
\MMMt'(\rho) =  \frac{\sqrt{3}}{2} \xstar \frac{-\kappa(\rho) \kappa'(\rho)}{\sqrt{1 - \left(\kappa(\rho)\right)^2}} + \kappa - \left(\hal \xstar - \rho \right) \kappa'(\rho) = \kappa(\rho) - \kappa'(\rho) \left(\hal \xstar - \rho + \frac{\sqrt{3} \xstar \kappa(\rho)}{2 \sqrt{1 - \left(\kappa(\rho)\right)^2}}  \right).
\end{equation*}
Since $\rho \geq \rho_0$, we have $\alphamax(\rho) = \alphac$ and $\kappa(\rho)$ is given, see \eqref{def:kappar}, by:
\begin{equation*}
\kappa(\rho) = \left(\sqrt{3.01} \xstar e^{-\frac{\pi}{2 \alphac} \xstar^2}\right) \times  \frac{1}{\rho} e^{\frac{\pi}{2 \alphac} \rho^2}, \quad \kappa'(\rho) = \frac{1}{\rho} \left( \frac{\pi}{\alphac} \rho^2 -1  \right) \kappa(\rho).
\end{equation*}
After factoring out $\kappa(\rho)$, which is positive, we are left to show that:
\begin{equation*}
1 - \frac{1}{\rho} \left( \frac{\pi}{\alphac} \rho^2 -1  \right) \left(\hal \xstar - \rho + \frac{\sqrt{3} \xstar \kappa(\rho)}{2 \sqrt{1 - \left(\kappa(\rho)\right)^2}}  \right) \geq 0.
\end{equation*}
We now use some simple bounds: $\frac{1}{\rho} \leq \frac{1}{\rho_0}$, $\frac{\pi}{\alphac} \rho^2 \leq \frac{\pi}{\alphac} \left( \frac{\xstar}{2} \right)^2$, $\hal \xstar - \rho \leq \hal \xstar - \rho_0$, $\frac{\kappa(\rho)}{\sqrt{1 - \left(\kappa(\rho)\right)^2}} \leq \frac{\kappa\left( \frac{\xstar}{2} \right)}{\sqrt{1 - \left(\kappa\left( \frac{\xstar}{2} \right)\right)^2}}$ (because $\kappa$ is increasing), and we obtain the purely numerical inequality:
\begin{equation*}
1 - \frac{1}{\rho_0} \left( \frac{\pi}{\alphac} \left( \frac{\xstar}{2} \right)^2  -1  \right) \left(\hal \xstar - \rho_0 + \frac{\sqrt{3} \xstar}{2} \frac{\kappa\left( \frac{\xstar}{2} \right)}{\sqrt{1 - \left(\kappa\left( \frac{\xstar}{2} \right)\right)^2}} \right) \geq 0,
\end{equation*}
which is true. 

\item If $\frac{\xstar}{2} \leq \rho \leq \frac{\xstar}{\sqrt{3}}$, we have:
\begin{equation*}
\MMMt(\rho)  = \left( \frac{\sqrt{3}}{2} \xstar \sqrt{1 - \left(\kappa(\rho)\right)^2} - \left(\rho- \hal \xstar\right) \kappa(\rho)  \right), \quad \exp\left(\frac{\pi}{\alphamax(\rho)} \left( \xstar \rho - \rho^2  \right) \right) = \exp\left(\frac{\pi}{\alphac} \left( \xstar \rho - \rho^2  \right) \right).
\end{equation*}
Since $\rho \geq \frac{\xstar}{2}$, we have $\left( \xstar \rho - \rho^2  \right)' \leq 0$, so the exponential term is decreasing. On the other hand $\MMMt$ is also decreasing, because $\kappa$ increases.
\end{itemize}
From this study, we get that $\Im$ is decreasing on $[0, \rho_0]$, then increasing on $[\rho_0, \hal \xstar]$, and then decreasing again on $[\hal \xstar, \frac{\xstar}{\sqrt{3}}]$. 
Numerically, one finds
\begin{equation*}
\Im(\rho_0) > \Im\left( \frac{\xstar}{\sqrt{3}} \right) \approx 3.581 \text{ versus } 3 \xstar^2 \approx 3.464.
\end{equation*}
\end{proof}

\subsubsection*{Step 7. Conclusion.}
Returning to \eqref{def:Fm}, we have found that for all admissible $u, \alpha, v$, we have:
\begin{equation*}
|(s_5+u) \cdot v|^2 e^{- \frac{2\pi}{\alpha} s_5 \cdot u - \frac{\pi}{\alpha} |u|^2} \geq 3 \xstar^2 + 0.1
\end{equation*}
In view of \eqref{Psiasu0debutcomplique}, \eqref{psiasu0uvsmallB}, we can then guarantee that:
\begin{equation*}
\Psisas(u) - \Psisas(0) \geq 0.1 e^{- \frac{\pi}{\alpha} \xstar^2}. 
\end{equation*}
Since $u \in \HH$ is bounded, we may write $|u|^2 \leq \frac{1}{c}0.1$ for $\cc$ small enough, and we obtain \eqref{Psisasu0}.
\end{proof}

\subsection{Higher shells}
We conclude this section with a quick study of higher shells. We claim that if $S$ is shell of radius $r \geq  \xss := \xstar \sqrt{3}$ (the second smallest non-zero distance between lattice points), we have:
\begin{equation}
\label{HighShellsBig}
\sum_{s \in S} |(s+u) \cdot v|^2 e^{- \frac{\pi}{\alpha} |s+u|^2} - 3 r^2 e^{- \frac{\pi}{\alpha} r^2} \geq 0.
\end{equation}
We also split the discussion into two cases, which are easier to treat than their “first shell” counterparts.
\begin{lemma}[$\frac{\pi}{\alpha}|u|$ small]
\label{lem:HigherShellsusmall}
Assume that $u, \alpha$ are such that:
\begin{equation}
\label{eq:uvsmallulargeHigh}
\alpha \leq \alphac, \quad \frac{\pi}{\alpha} |u| \leq \muh := 1.5.
\end{equation}
Then \eqref{HighShellsBig} holds.
\end{lemma}
\begin{proof}
We follow the same strategy of proof as for Lemma \ref{lem:pialphauverysmall}: perform a Taylor expansion to third order and use the symmetries of the hexagon. Replacing $\xstar$ by $r$ in \eqref{expandsquareSDL}, \eqref{ExpandSUM}, \eqref{FDLT2}, we obtain that \eqref{HighShellsBig} holds if:
\begin{equation*}
3r^2 + \frac{3}{2} \K^2 r^4 |u|^2 + |u \cdot v|^2 \left(6 + 3 \K^2 r^4 \left(1 - 2 \K |u|^2 \right) + 6 \K r^2 \left( \K |u|^2 - 2 \right) \right) \geq 3 r^2 e^{\K |u|^2}.
\end{equation*}
We still have the constraints $\K \geq \K_0 := \frac{\pi}{\alphac}$, and $\mu := \K |u| \leq \muh$. The good news is that the quantity within the parenthesis is now non-negative: write it as
\begin{equation*}
6 + 3 \K^2 r^4 - 6 \K r^4 \mu^2 + 6 \mu^2 r^2 - 12 \K r^2 \geq  6 + 3 \K^2 r^4 - 6 \K r^4 \muh^2 + 6 \muh^2 r^2 - 12 \K r^2
\end{equation*}
the critical point of this convex quantity $\K = \muh^2 + \frac{2}{r^2} \leq \K_0$, so the minimal value is at $\K_0$, furthermore:
\begin{equation*}
\frac{\dd}{\dd r} \left(6 + 3 \K_0^2 r^4 - 6 \K_0 r^4 \muh^2 + 6 \muh^2 r^2 - 12 \K_0 r^2\right) \geq 0 \iff r^2 \geq \frac{2\K_0 - \muh^2}{\K_0^2 - 2\K_0 \muh^2},
\end{equation*} 
which is true for $r \geq \xss := \sqrt{3}\xstar$, and we can finally check a purely numerical inequality:
\begin{equation*}
6 + 3 \K_0^2 r^4 - 6 \K_0 r^4 \muh^2 + 6 \muh^2 r^2 - 12 \K_0 r^2 \geq 6 + 3 \K_0^2 \xss^4 - 6 \K_0 \xss^4 \muh^2 + 6 \muh^2 \xss^2 - 12 \K_0 \xss^2 \geq 0.
\end{equation*}

We are thus left to prove that, for all admissible $\K, u$,
\begin{equation*}
1+ \frac{1}{2} \K^2 r^2 |u|^2  \geq e^{\K |u|^2}, \text{ i.e. } 1+ \frac{1}{2} \mu^2 r^2 \geq e^{\frac{\mu^2}{\K}} 
\end{equation*}
which again boils down to an explicit computation: the worst case for $\K$ is at $\K = \K_0$, and the quantity $1+ \frac{1}{2} \mu^2 r^2 - e^{\frac{\mu^2}{\K_0}}$ is concave in $\mu^2$, vanishes at $\mu^2 = 0$, so it remains to check that:
\begin{equation*}
1+ \frac{1}{2} \muh^2 r^2 - e^{\frac{\muh^2}{\K_0}} \geq 0, 
\end{equation*}
which is true with our choices $\alphac = 0.552$ (and $\K_0 = \frac{\pi}{\alphac}$), $\muh = 1.5$.
\end{proof}

\begin{lemma}[$\frac{\pi}{\alpha}|u|$ large]
\label{lem:HigherShellsularge}
Assume that $u, \alpha$ are such that:
\begin{equation*}
\frac{\pi}{\alpha} |u| \geq \muh := 1.5
\end{equation*}
Then \eqref{HighShellsBig} holds.
\end{lemma}
\begin{proof}
We follow a similar strategy as for the proof of Lemma \ref{lem:uvsmallularge}. Returning to the situation shown in Figure \ref{figures4}, we now focus on the contributions of $s_5$ \emph{and} $s_4$ together. The first key observation is that, taking $\theta \in \left[0, \frac{\pi}{6}\right]$ as in Figure \ref{figures4}, we have both $\pi - \theta$ and $\pi + \frac{\pi}{3} - \theta$ in $\left[\pi - \frac{\pi}{3}, \pi + \frac{\pi}{3}\right]$, hence:
\begin{equation*}
s_4 \cdot u \leq - \hal r |u|, \quad s_5 \cdot u \leq - \hal r |u|.
\end{equation*}
The contributions from $s_4$ and $s_5$ are thus bounded below by:
\begin{equation*}
|(s_5+u) \cdot v|^2 e^{- \frac{\pi}{\alpha} |s_5+u|^2} + |(s_4+u) \cdot v|^2 e^{- \frac{\pi}{\alpha} |s_4+u|^2} \geq \left(|(s_5+u) \cdot v|^2  + |(s_4+u) \cdot v|^2 \right) e^{- \frac{\pi}{\alpha} r^2 - \frac{\pi}{\alpha} |u|^2 + \frac{\pi}{\alpha} r |u|},
\end{equation*}
and to get \eqref{HighShellsBig} it suffices to prove:
\begin{equation}
\label{eq:toprove_finally}
\left(|(s_5+u) \cdot v|^2  + |(s_4+u) \cdot v|^2 \right) e^{\frac{\pi}{\alpha} \left(r |u| - |u|^2\right)} \geq 3r^2.
\end{equation}
\begin{claim}
We have
\begin{equation*}
|(s_5+u) \cdot v|^2  + |(s_4+u) \cdot v|^2 \geq r^2 \left( \hal - \frac{3|u|^2}{r^2} \right)
\end{equation*}
\end{claim}
\begin{proof}
Write $v = (\cos \phi, \sin \phi)$. Expanding the squares, using that $(s_5 + s_4) \cdot v = \sqrt{3} r \cos\left(\pi + \frac{\pi}{6} - \phi \right)$ (see Figure \ref{figures4}), and that $(u \cdot v) \geq - |u|$, we get:
\begin{multline*}
|(s_5+u) \cdot v|^2  + |(s_4+u) \cdot v|^2 = r^2 \left(\cos^2(\pi - \phi) + \cos^2(\pi + \frac{\pi}{3} - \phi) \right) + 2 \left((s_5 + s_4) \cdot v \right) (u \cdot v) + 2 |u \cdot v|^2 \\ 
\geq r^2 \left( \cos^2(\phi) + \cos^2\left(\phi - \frac{\pi}{3}\right) \right) - 2 \sqrt{3} |u| r \left|\cos\left(\phi - \frac{\pi}{6} \right)\right|  = r^2 \left( \hal + \cos^2\left(\phi - \frac{\pi}{6}\right) \right) - 2 \sqrt{3} |u|r \left|\cos\left(\phi - \frac{\pi}{6} \right)\right|.
\end{multline*}
The minimum of $c \mapsto r^2 \left( \hal + c^2 \right) - 2 \sqrt{3} |u| r c$ for $c \in [0,1]$ is $\frac{r^2}{2} - 3 |u|^2$.
 \end{proof}
In order to get \eqref{eq:toprove_finally}, it remains to show that:
\begin{equation*}
\left( \hal - \frac{3|u|^2}{r^2} \right) e^{\frac{\pi}{\alpha} \left(r |u| - |u|^2\right)} \geq 3.
\end{equation*}
The worst possibility for $r$ is the smallest one, namely $r = \xss := \sqrt{3}\xstar$. Let  \renewcommand{\FF}{\mathrm{F}}
\begin{equation*}
\FF(\alpha, \rho) := \left( \hal - \frac{\rho^2}{\xstar^2} \right) e^{\frac{\pi}{\alpha} \left(\sqrt{3} \xstar \rho - \rho^2\right)}.
\end{equation*}
A direct computation shows that for fixed $\alpha$, $\log \FF(\alpha, \cdot)$ is strictly concave in $\rho$ and thus attains its minimum at the endpoints of the admissible interval defined by the conditions $\frac{\pi}{\alpha} \rho \geq \muh$ and $\rho \leq \frac{\xstar}{\sqrt{3}}$, i.e. at $\rho = \frac{\muh \alpha}{\pi}$ or at $\rho = \frac{\xstar}{\sqrt{3}}$. With our choices of $\alphac = 0.552, \muh = 1.5$, we have on the one hand:
\begin{equation*}
\FF\left(\alpha, \frac{\xstar}{\sqrt{3}}\right) \geq \FF\left(\alphac, \frac{\xstar}{\sqrt{3}}\right) = \frac{1}{6} e^{\frac{\pi}{\alphac} \frac{4}{3 \sqrt{3}}} > 3,
\end{equation*}
and on the other hand:
\begin{equation*}
\FF\left(\alpha, \frac{\muh \alpha}{\pi}\right) = \left( \hal - \frac{\muh^2 \alpha^2}{\pi^2 \xstar^2} \right) e^{\sqrt{3} \xstar \muh - \frac{\muh^2 \alpha}{\pi}}  \geq \left( \hal - \frac{\muh^2 \alphac^2}{\pi^2 \xstar^2} \right) e^{\sqrt{3} \xstar \muh - \frac{\muh^2 \alphac}{\pi}} > 3.
\end{equation*}
\end{proof}

\begin{longtable}{ll}
\caption{Table of Notations and Constants}\\
\hline
$\AD$ & The hexagonal lattice (of density 1). \\
$\HH$ & Fundamental domain (Voronoi cell) of $\AD$. \\
$\ADD$ & The reciprocal lattice of $\AD$. \\
$\Omega$ & Fundamental domain of $\ADD$ (Pontryagin dual). \\
$\Ss$ & The first shell (6 nearest neighbors) of $\AD$. \\
$\xstar$ & Minimal distance of $\AD$, $\sqrt{2/\sqrt{3}}$. \\
$\xss$ & The second smallest non-zero distance, $\xstar \sqrt{3}$.\\
c.m.s.d. & Completely monotonic function of square distance. \\
$\Pa$ & Gaussian potential, $r \mapsto e^{-\pi \alpha r^2}$. \\
$\Ga(x)$ & Gaussian function on $\R^2$, $\Ga(x) = \Pa(|x|)$. \\
$\muf$ & Bernstein measure for a c.m.s.d. function $f$. \\
$\Ef(\bX)$ & The $f$-energy per point of $\bX$. \\
$\Ea(\bX)$ & The $\Pa$-energy of $\bX$.\\
$\pP$ & A perturbation map $\pP: \AD \to \R^2$. \\
$\AD + \pP$ & The perturbed lattice configuration. \\
$\|\pP\|$ & Size of the perturbation, $\sup_x |\pP(x)|$. \\
$\Px$ & Probability measure of relative displacements. \\
$\CC_x$ & Covariance matrix of relative displacements. \\
$\RR_x$ & Autocorrelation matrix of the perturbation field. \\
$\FS(\PP)$ & First-Shell size of $\PP$, $\sum_{x \in \Ss} \int |u|^2 d\Px$. \\
$\Cs$ & The $\STP$-valued spectral measure of $\pP$. \\
$\tCs$ & The trace measure $\Cs^{11} + \Cs^{22}$. \\
$\Cst$ & The trace derivative, $d\Cs / d\tCs$. \\
$\SM(\PP)$ & Spectral Measure size of $\PP$, $\Tr(\int |\omega|^2 d\Cs)$. \\
$\Psiav(u)$ & Auxiliary lattice sum $\sum |(x+u) \cdot v|^2 e^{-\frac{\pi}{\alpha}|x+u|^2}$. \\
$\Psisas(u)$ & Contribution to $\Psiav(u)$ from $\Ss \cup \{0\}$. \\
$\alphaL$ & Threshold $\alpha$ separating large/small regimes. \\
$\alphac$ & Numerical constant, $0.552$. \\
$\mustar$ & Numerical constant, $2.73$. \\
$\muh$ & Numerical constant, $1.5$.\\
\end{longtable}

\clearpage
\bibliographystyle{alpha}
\bibliography{LocalOpt_abbrev_cleaned}
\end{document}